\documentclass{article}
\usepackage{amsmath,amsthm,amssymb,amscd}
\usepackage[all,cmtip,color]{xy}
\usepackage[german,frenchb,english]{babel}
\usepackage{amsmath}
\usepackage{amssymb}
\usepackage[backref]{hyperref}
\usepackage{mathrsfs} 
\usepackage{amssymb}
\usepackage{todonotes}
\usepackage{amssymb}
\usepackage[all,cmtip]{xy}
\input diagxy
\usepackage[backref]{hyperref}
\usepackage[margin = 2.5 cm]{geometry}

\DeclareMathOperator{\N}{\mathsf{N}}
\DeclareMathOperator{\fl}{\mathsf{fl}}
\DeclareMathOperator{\DMod}{\mathsf{DMod}}
\DeclareMathOperator{\APerf}{\mathsf{APerf}}
\DeclareMathOperator{\cn}{\mathsf{cn}}
\DeclareMathOperator{\Hom}{\mathsf{Hom}}

\DeclareMathOperator{\St}{\mathsf{st}}
\DeclareMathOperator{\qcqs}{\mathsf{qcqs}}
\DeclareMathOperator{\Rings}{\mathsf{Rng}}
\DeclareMathOperator{\DK}{\mathsf{DK}}
\DeclareMathOperator{\AbGrp}{\mathsf{AbGrp}}

\DeclareMathOperator{\A}{\mathsf{A}}

\DeclareMathOperator{\Grpd}{\infty-\mathsf{Gpd}}

\renewcommand{\P}{\mathsf{P}}

\newcommand{\Hc}{\mathcal{H}}

\DeclareMathOperator{\Vect}{Vect}

\DeclareMathOperator{\colim}{\mathsf{colim}}
\renewcommand{\lim}{\mathsf{lim}}

\DeclareMathOperator{\Perf}{\mathsf{Perf}}
\DeclareMathOperator{\coker}{coker}

\newcommand{\C}{\mathsf{C}}

\newcommand\triplerightarrow{\mathrel{\substack{\textstyle\rightarrow\\[-0.6ex]
                      \textstyle\rightarrow \\[-0.6ex]
                      \textstyle\rightarrow}} }

\DeclareMathOperator{\id}{id}

\DeclareMathOperator{\Frac}{Frac}

\DeclareMathOperator{\Set}{Set}

\DeclareMathOperator{\Mod}{\mathsf{Mod}}
\DeclareMathOperator{\Gpd}{\mathsf{Grpd}}

\DeclareMathOperator{\Loc}{\mathsf{Loc}}

\newcommand{\Mm}{\mathsf{M}}

\DeclareMathOperator{\Ind}{\mathsf{Ind}}
\DeclareMathOperator{\Fun}{\mathsf{Fun}}

\DeclareMathOperator{\op}{\mathsf{op}}

\DeclareMathOperator{\Ab}{\mathbb{A}}
\DeclareMathOperator{\Aff}{Aff}

\DeclareMathOperator{\QCoh}{\mathsf{QCoh}}
\DeclareMathOperator{\QC}{\mathsf{QC}}

\DeclareMathOperator{\Coh}{\mathsf{Coh}}

\DeclareMathOperator{\Cat}{\mathsf{Cat}}
\DeclareMathOperator{\inftyCat}{\infty-\mathsf{Cat}}

\DeclareMathOperator{\F}{\mathcal{F}}

\DeclareMathOperator{\image}{im}

\DeclareMathOperator{\Map}{\mathsf{Map}}
\DeclareMathOperator{\G}{\mathbb{G}}

\DeclareMathOperator{\GL}{GL}

\DeclareMathOperator{\M}{M}
\DeclareMathOperator{\Ob}{\mathbb{O}}

\DeclareMathOperator{\Y}{\mathcal{Y}}

\DeclareMathOperator{\D}{\mathsf{D}}

\DeclareMathOperator{\Sch}{\mathsf{Sch}}

\DeclareMathOperator{\Gg}{\mathcal{G}}

\DeclareMathOperator{\Spec}{Spec}

\DeclareMathOperator{\Hhom}{\underline{\Hom}}
\DeclareMathOperator{\Oo}{\mathcal{O}}

\begin{document}
\newtheorem{definition}{Definition}[section]
\newtheorem{theorem}[definition]{Theorem}
\newtheorem{proposition}[definition]{Proposition}
\newtheorem{corollary}[definition]{Corollary}
\newtheorem{conj}[definition]{Conjecture}
\newtheorem{lemma}[definition]{Lemma}
\newtheorem{rmk}[definition]{Remark}
\newtheorem{cl}[definition]{Claim}
\newtheorem{example}[definition]{Example} 
\newtheorem{claim}[definition]{Claim}
\newtheorem{ass}[definition]{Assumption}
\newtheorem{warning}[definition]{Dangerous Bend}
\newtheorem{porism}[definition]{Porism}

\author{Michael Groechenig}

\title{Adelic Descent Theory\let\thefootnote\relax\footnotetext{\url{m.groechenig@ic.ac.uk}, Department of Mathematics, Imperial College London, UK}}
\maketitle 

\abstract{A result of Andr\'e Weil allows one to describe rank $n$ vector bundles on a smooth complete algebraic curve up to isomorphism via a double quotient of the set $\mathrm{GL}_n(\mathbb{A})$ of regular matrices over the ring of ad\`eles (over algebraically closed fields, this result is also known to extend to $G$-torsors for a reductive algebraic group $G$). In the present paper we develop analogous adelic descriptions for vector and principal bundles on arbitrary Noetherian schemes, by proving an adelic descent theorem for perfect complexes. We show that for Beilinson's co-simplicial ring of ad\`eles $\mathbb{A}_X^{\bullet}$, we have an equivalence $\mathsf{Perf}(X)  \simeq |\mathsf{Perf}(\mathbb{A}_X^{\bullet})|$ between perfect complexes on $X$ and cartesian perfect complexes for $\mathbb{A}_X^{\bullet}$. Using the Tannakian formalism for symmetric monoidal $\infty$-categories, we conclude that a Noetherian scheme can be reconstructed from the co-simplicial ring of ad\`eles. We view this statement as a scheme-theoretic analogue of Gelfand--Naimark's reconstruction theorem for locally compact topological spaces from their ring of continuous functions. Several results for categories of perfect complexes over (a strong form of) flasque sheaves of algebras are established, which might be of independent interest.}

\tableofcontents
 
\section*{Introduction}

For a smooth, irreducible, complete, algebraic curve $X$, we denote by $F$ the field of rational functions, by $\Ob$ the product $\prod_{x \in X_{cl}} \widehat{\Oo}_x$ ranging over closed points, and 
$$\Ab = \sideset{}{'}\prod_{x\in X_{cl}} \widehat{F}_x = \{(f_x)_{x \in X_{cl}}|f_x \in \widehat{\Oo}_x\text{ for almost all }x\}.$$
This object is called the ring of ad\`eles. Andr\'e Weil was probably the first to appreciate the close connection between ad\`eles and the geometry of curves (see the letter \cite{weil1938algebraischen} to Hasse where the case of line bundles is discussed, and \cite{zbMATH03028558} for the closely related notion of matrix divisors).

\begin{theorem}[Weil]\label{thm:weil}
Let $X$ be an algebraic curve, defined over an algebraically closed field $k$, and let $G$ be a reductive algebraic group. We then have an equivalence between the groupoid of $G$-torsors on $X$, $BG(X)$, and the groupoid defined by the double quotient $[G(F)\setminus G(\Ab) \;/\; G(\Ob)].$
\end{theorem}

Weil's theorem is central to the Geometric Langlands Programme, as it connects the arithmetic conjectures to their geometric counterpart. For a survey of this connection see \cite{frenkel2007lectures}. The interplay of Weil's result with conformal field theory is discussed by Witten \cite[Section V]{witten1988quantum}. 

In this article we present a generalisation of Weil's theorem to arbitrary Noetherian schemes. We will deduce it from an adelic descent result for the perfect complexes. 
The co-simplicial ring $\Ab_X^{\bullet}$ was introduced by Beilinson in \cite{MR565095}, as a generalisation of the theory of ad\`eles for curves. A similar construction has also been obtained by Parshin for algebraic surfaces. If $X$ is a curve, the co-simplicial ring is given by the diagram
\begin{equation*}
\xymatrix@C=1em{
F \times \mathbb{O}_X \ar@<-0.7ex>[r] \ar@<0.7ex>[r] & \mathbb{A}_X \times F \times \mathbb{O}_X \ar[l] \ar@<-1.4ex>[r] \ar[r] \ar@<1.4ex>[r] & \ar@<-0.7ex>[l] \ar@<0.7ex>[l]\cdots,
}
\end{equation*}
which captures the adelic rings $F$, $\Ob_X$, and $\Ab_X$, and the various maps between them, used to formulate Weil's Theorem \ref{thm:weil}.

\begin{theorem}[Adelic Descent]\label{thm:main}
Let $X$ be a Noetherian scheme. We denote by $\Ab_X^{\bullet}$ Beilinson's co-simplicial ring of ad\`eles (see Definition \ref{defi:adeles}). We have an equivalence of symmetric monoidal $\infty$-categories $\Perf(X)_{\otimes}  \simeq |\Perf(\Ab_X^{\bullet})_{\otimes}|$, where the right hand side denotes the $\infty$-category of cartesian $\Ab_X^{\bullet}$-modules.
\end{theorem}

This theorem also holds for almost perfect complexes, as we show in Corollary \ref{cor:APerf}.
According to Lieblich, the study of perfect complexes is the mother of all moduli problems (see the abstract of \cite{lieblich2005moduli}). The Tannakian formalism enables us to make this philosophical principle precise. Using the results of Bhatt \cite{Bhatt:2014aa} and Bhatt--Halpern-Leistner \cite{bhatt2015tannaka}, we obtain a descent result for $G$-torsors (we may replace $BG$ by more general algebraic stacks).

\begin{theorem}\label{thm:main2}
Let $X$ be a Noetherian scheme, the geometric realisation of the simplicial affine scheme $\Spec \Ab_X^{\bullet}$ in the category of Noetherian algebraic stacks with quasi-affine diagonal, is canonically equivalent to $X$. In particular, we have $BG(X)  \simeq |BG(\Spec \Ab_X^{\bullet})|$, if $G$ is a Noetherian affine algebraic group scheme.
Let $G$ be a special group scheme (for example $G = \GL_n$). We denote by $G(\Ab_X^1)^{\text{cocycle}}$ the subset consisting of $\phi \in \G(\Ab_X^1)$ satisfying the cocycle condition $\phi_{02} = \phi_{01}\circ \phi_{12}$ in $G(\Ab_X^2)$. There is an equivalence of groupoids
$BG(X)  \simeq [G(\Ab_X^1)^{\text{cocycle}}/G(\Ab_X^0)].$
\end{theorem}

In characteristic $0$, the assumption that $G$ be Noetherian can often be dropped. We refer the reader to Corollary \ref{perfect_reconstruction}. We refer the reader to Paragraph \ref{specialcase} for a more detailed discussion of the adelic description of $G$-bundles on Noetherian schemes $X$. The case of punctured surfaces has also been considered by Garland and Patnaik in \cite{garlandgeometry}. In \cite{MR697316}, Parshin used adelic cocycles for $G$-bundles as above, to obtain formulae for Chern classes in adelic terms.

As a further consequence of the adelic descent formalism, we obtain an analogue of Gelfand--Naimark's reconstruction theorem for locally compact topological spaces \cite{gelfand}. Recall that \emph{loc. cit.} shows that a locally compact topological space can be reconstructed from the ring of continuous functions. It is well-known that a similar result cannot hold for non-affine schemes. However, our result implies that a Noetherian scheme $X$ can be reconstructed from the co-simplicial ring of ad\`eles.

\begin{theorem}\label{thm:main4}
The functor $\Ab^{\bullet}\colon \Sch^{N} \hookrightarrow (\Rings^{\Delta})^{\op}$ from the category of Noetherian schemes to the dual category of co-simplicial commutative ring, has an explicit left-inverse, sending $R^{\bullet}$ to $|\Spec R^{\bullet}|$.
\end{theorem} 

It is instructive to meditate on the differences and similarities with Gelfand--Naimark's theorem. While their result gets by with plain rings, our Theorem \ref{thm:main4} requires a diagram of rings (see Corollary \ref{cor:schematic_reconstruction} for a precise statement to which extent the co-simplicial structure is needed). However, the necessary condition of local compactness for topological spaces is not unlike the restriction that the scheme be Noetherian. 

For a quasi-compact and quasi-separated scheme $X$ we may choose a finite cover by affine open subschemes $\{U_i\}_{i = 1,\dots,n}$. The coproduct $U = \coprod_{i = 1}^nU_i$ is then still an affine scheme, and we have a map $U \rightarrow X$. Choosing a finite affine covering for $U \times_X U$, and iterating this procecure, we arrive at a simplicial affine scheme $U_{\bullet} \rightarrow X$, which yields a hypercovering of $X$. The coordinate ring yields a co-simplicial ring $\Gamma(U_{\bullet})$ associated to $X$. However this construction is a priori not functorial, since it depends on the chosen coverings. Nonetheless, using the construction $X \mapsto X^Z$ introduced by Bhatt and Scholze \cite{bhatt2013pro}, we obtain another functor as in Theorem \ref{thm:main4} (the author thanks Bhargav Bhatt for bringing this to his attention). 

Our Theorem \ref{thm:main} relies heavily on Beilinson's \cite{MR565095}, which constructs a functor, sending a quasi-coherent sheaf $\F$ on $X$ to an $\Ab_X^{\bullet}$-module $\Ab_X^{\bullet}(\F)$. Beilinson observes that the latter co-simplicial module gives rise to a chain complex, computing the cohomology of $\F$. This chain complex can be obtained by applying the Dold-Kan correspondence, or taking the alternating sum of the face maps in each degree: $[\Ab^0(\F) \to^{\partial_0 - \partial_1} \Ab^1(\F) \cdots]$. Beilinson's result can be stated as:
$$H^i(X,\F)  \simeq H^i(\DK(\Ab_X^{\bullet}(\F))).$$
The reason is that the sheaves $\A_X^k\colon U \mapsto \Ab_U^{k}(\F)$ are flasque, and hence it remains to show that the corresponding complex of sheaves defines a flasque resolution of $\F$. The details are explained in \cite{MR1138291}. Since morphisms in $\Perf(X)$ are closely related to sheaf cohomology, it is not difficult to deduce from Beilinson's observation the existence of a fully faithful functor $\Perf(X) \hookrightarrow |\Perf(\Ab_X^{\bullet})|$, amounting to a sort of a cohomological descent result. Our proof of the Adelic Descent Theorem is hence mainly concerned with establishing that this functor is essentially surjective, that is, establishing descent for objects in those $\infty$-categories. In heuristic terms, our theorem asserts that a perfect complex $M$ on $X$ can be described by an iterative formal glueing procedure from the adelic parts $\Ab_X^{\bullet}(M)$. 

Our main theorem uses the language of stable $\infty$-categories. Replacing $\Perf(X)$ and $\Perf(\Ab_X^{\bullet})$ by their homotopy categories would render the result wrong. However, the theorem could be formulated in more classical language. The $\infty$-category of cartesian perfect modules over a co-simplicial ring, such as $\Perf(\Ab_X^{\bullet})$, has a model, as is discussed in \cite[1.2.12]{MR2394633} by To\"en--Vezzosi.

Since the construction of $\Ab_X^{\bullet}$ involve iterative completion and localisation procedures, the formal descent result of Beauville--Laszlo \cite{MR1320381}, and Ben-Bassat--Temkin \cite{MR3003930} (for quasi-coherent sheaves in each case) are closely related. These theorems allow one to glue sheaves on a scheme $X$ with respect to the formal neighbourhood of a closed subvariety $Y$, and its open complement $X \setminus Y$. Beauville--Laszlo developed such a descent theory for an affine scheme $X$, and a closed subvariety $Y$ given by a principal ideal. This result was motivated by the study of conformal blocks \cite{MR1289330}.
The second article does not require the restrictions of $X$ to be affine, and $Y$ to be principal, however utilises the theory of Berkovich spaces, in order to formulate the glueing result. 
Our Theorem \ref{thm:main} gives a very similar descent theory, but uses all closed subvarieties at once and avoids rigid geometry.
 
One of the key properties of the ad\`eles allowing one to establish effectivity of adelic descent data is a strengthening of the theory of flasque sheaves of algebras. 

\begin{theorem}\label{thm:main3}
Let $X$ be a quasi-compact topological space, and $\A$ a l\^ache sheaf of algebras (see Definition \ref{defi:lache}), with ring of global sections $R = \Gamma(\A)$. The global section functor $\Gamma\colon \Mod(\A) \rightarrow \Mod(\A)$ restricts to a symmetric monoidal equivalence $\Perf(\A)_{\otimes}  \simeq \Perf(R)_{\otimes}$.
\end{theorem}

We show in Lemma \ref{lemma:adeles_section} that the ad\`eles $\Ab_X^k$ are l\^ache sheaves of algebras. The derived equivalence underlying our theorems decomposes into two parts:
$$\Perf(X)  \simeq |\Perf(\A_X^{\bullet})|  \simeq |\Perf(\Ab_X^{\bullet})|.$$
The second equivalence is deduced from Theorem \ref{thm:main3}. The first equivalence can be established by local verifications. 
\\
\\
\noindent\textit{Acknowledgements:} The author first conjectured the main theorem during a conversation with Bertrand To\"en. I thank him for his interest and encouragement, which subsequently led to the present work. Moreover I thank H\'el\`ene Esnault and Gabriele Vezzosi for interesting discussions while this article was prepared. I am grateful to Dennis Gaitsgory for alerting me of an erroneous lemma contained in a preliminary version of this paper. Through joint work with Oliver Br\"aunling and Jesse Wolfson I was introduced to the fascinating world of ad\`eles and its curious interplay with algebraic geometry. It is a pleasure to thank them for all the mathematics they taught me. Their influence on my thinking permeates the entire text. I am grateful for EPSRC funding received under P/G06170X/1.

\section{A reminder of Beilinson-Parshin ad\`eles}

In \cite{MR565095} Beilinson generalised the notion of ad\`eles to arbitrary Noetherian schemes, and studied the connection ad\`eles bear with coherent cohomology. We will briefly review his definition, and the main properties of relevance to us. Except for the assertion that ad\`eles are flasque sheaves (Corollary \ref{cor:adeles_flasque}), we will not provide a proof for those statements, and refer the reader instead to Huber's \cite{MR1138291}. Examples can be found in Morrow's survey article about ad\`eles and their relation to higher local fields \cite{morrow}.

\subsection{Recollection}

Henceforth we denote by $X$ a Noetherian scheme, with underlying topological space $|X|$ and structure sheaf $\Oo_X$.

\begin{definition}\label{defi:X_k}
Let $X$ be a scheme with underlying topological space $|X|$. For $x,y \in |X|$ we write $x \leq y$ for $x \in \overline{\{y\}}$; this defines a partial ordering on the set $|X|$. We denote the set $\{(x_0,\dots,x_k)\in |X|^{\times k+1}|x_0\leq \cdots \leq x_k\}$ by $|X|_k$. 
\end{definition}

One sees that $|X|_k$ is in fact the set of $k$-simplices of a simplicial set $|X|_{\bullet}$. This simplicial structure will be reflected in a co-simplicial structure for Beilinson-Parshin ad\`eles. 

\begin{definition}\label{defi:X_simplicial}
The simplicial set $|X|_{\bullet}\colon \Delta^{\op} \rightarrow \Set$ is defined to be the functor, sending $[n] \in \Delta^{\op}$ to the set of ordered maps $[n] \rightarrow |X|$, with respect to the partial ordering $\leq$ on $X$ defined in Definition \ref{defi:X_k}.
\end{definition}

Following \cite{MR565095} we define ad\`eles with respect to a subset $T \subset |X|_k$. The case of interest to us, will be $T = |X|_k$, but the recursive nature of the definition necessitates a definition for general subsets $T \subset |X|_k$. We begin with the following preliminary definitions.

\begin{definition}\label{defi:pre_adeles}
Let $X$ be a Noetherian scheme, and $k \in \mathbb{N}$ a non-negative integer.
\begin{itemize}
\item[(a)] For $x \in |X|$ and $T \subset |X|_k$ we define ${}_xT = \{(x_0 \leq \cdots \leq x_{k-1}) \in |X|^{\times k-1}|(x_0 \leq \cdots \leq x_{k-1} \leq x) \in T\}$.
\item[(b)] For $x \in |X|$ we denote by $\Oo_x$ the local ring at $x$ with maximal ideal $\mathfrak{m}_x$. We have the canonical morphism $j_{rx}\colon\Spec \Oo_x/\mathfrak{m}_x^r \rightarrow X$.
\end{itemize}
\end{definition}

It is convenient to define ad\`eles in a higher-dimensional situation as sheaves of $\Oo_X$-modules.

\begin{definition}\label{defi:adeles}
Let $X$ be a Noetherian scheme. The ad\`eles are the unique family of exact functors $\A_{X,T}(-)\colon \QCoh(X) \rightarrow \Mod(\Oo_X)$, indexed by subsets $T \subset |X|_k$, satisfying the following conditions:
\begin{itemize}
\item[(a)] The functor $\A_{X,T}(-)$ commutes with directed colimits.
\item[(b)] For $\F \in \Coh(X)$, and $k = 0$ we have $\A_{X,T}(\F) = \prod_{x \in T}\varprojlim_{r \geq 0} (j_{rx})_*(j_{rx})^*\F$.
\item[(c)] For $\F \in \Coh(X)$ and $k \geq 1$ we have $\A_{X,T}(\F) = \prod_{x\in |X|} \varprojlim_{r \geq 0}\A({}_xT,(j_{rx})_*(j_{rx})^*\F)$
\end{itemize}
\end{definition}

We refer the reader to \cite{MR1138291} for a detailed verification that the above family of functors is well-defined and exact. The ring of ad\'eles with respect to $T \subset |X|_n$ is defined by taking global sections of the sheaf of rings $\A_{X,T}(\Oo_X)$. Moreover it is important to emphasise that the sheaves of $\Oo_X$-modules $\A_{X,T}(\F)$ are in general not quasi-coherent.

\begin{definition}\label{defi:ring_adeles}
We denote the abelian group $\Gamma(X,\A_{X,T}(\F))$ by $\Ab_{X,T}(\F)$; and reserve the notation $\Ab_{X,T}$ for $\Ab_{X,T}(\Oo_X)$. By construction $\Ab_{X,T}(\F)$ is an $\Ab_{X,T}$-module.
\end{definition}

As we already alluded to, the most interesting case for us is when $T = |X|_k$. We reserve a particular notation for this situation.

\begin{definition}
We denote the sheaf $\A(|X|_k, \F)$ by $\A^k_X(\F)$. The abelian group $\Gamma(X,\A^k_X(\F))$ will be denoted by $\Ab_X^k(\F)$.
\end{definition}

As the superscript indicates, these sheaves can be assembled into a co-simplicial object. The proof of this can be found in \cite[Theorem 2.4.1]{MR1138291}.

\begin{proposition}\label{prop:co-simplicial}
Let $X$ be a Noetherian scheme, and $T_{\bullet} \subset |X|_{\bullet}$ a simplicial subset. There is an exact functor $\A_{X,T_{\bullet}}^{\bullet}\colon\QCoh(X) \rightarrow \Fun(\Delta,\Mod(\Oo_X))$, which commutes with directed colimits, and maps $([k],\F)$ to $\A_{X,T_k}(\F)$. We denote the functor $\Gamma(X,\A_{X,T_{\bullet}}^{\bullet}(-))\colon \QCoh(X) \rightarrow \Fun(\Delta,\AbGrp)$ by $\Ab_{X,T_{\bullet}}^{\bullet}(-)$; it is exact and commutes with directed colimits. The notation $\Ab_X^{\bullet}$ is reserved for the functor corresponding to the case $T_{\bullet} = |X|_{\bullet}$.
\end{proposition}

Let $X$ be an irreducible Noetherian scheme of dimension $1$, and $\F$ a coherent sheaf on $X$. We will discuss how the definitions above recover the classical theory of ad\`eles for algebraic curves.  Following classical conventions, we denote by 
$$\Ab_X(\F) = \sideset{}{'}\prod_{x\in |X|_{cl}} \widehat{\F}_x \otimes_{\widehat{\Oo}_x}\Frac \widehat{\Oo}_X,$$
the restricted product ranging over all closed points $x \in |X|_{cl}$. We denote by $\Ob_X(\F) = \prod_{x\in X_{cl}} \widehat{\F}_x$; and by $F_X(\F)$ the $\Oo$-module $\F_{\eta}$, where $\eta$ is the generic point of $X$. With respect to this notation we may identify the co-simplicial $\Oo$-module $\Ab_X^{\bullet}(\F)$ with 
\begin{equation*}
\xymatrix@C=1em{
F_X(\F) \times \mathbb{O}_X(\F) \ar@<-0.7ex>[r] \ar@<0.7ex>[r] & \mathbb{A}_X(\F) \times F_X(\F) \times \mathbb{O}_X(\F) \ar[l] \ar@<-1.4ex>[r] \ar[r] \ar@<1.4ex>[r] & \ar@<-0.7ex>[l] \ar@<0.7ex>[l]\cdots,
}
\end{equation*}
where $F_X(\F) \rightarrow \Ab_X(\F)$ is the diagonal inclusion, and $\Ob_X(\F) \rightarrow \Ab_X(\F)$ the canonical map. Embracing the usual redundancies in co-simplicial objects, that is the continual re-appearance of factors already seen at a lower degree level, we observe that Beilinson's $\Ab_X^{\bullet}$ captures the classical theory of ad\`eles. 

It is also helpful to understand the co-simplicial structure in the local case. Let $\sigma\colon [n] \rightarrow |X|_{\leq}$ be an element of $|X|_n$. We denote the ring of ad\`eles $\Ab_{X,\{\sigma\}}$ corresponding to $T = \{\sigma\} \subset |X|_n$ by $\Ab_{X,\sigma}$. Proposition \ref{prop:co-simplicial} implies that for every map $f\colon [m] \rightarrow [n]$ in $\Delta$ we have a ring homomorphism $\Ab_{X,\sigma \circ f} \rightarrow \Ab_{X,\sigma}$. The following assertion is also proven in \cite[Theorem 2.4.1]{MR1138291}.

\begin{lemma}\label{lemma:local_factors}
Let $X$ and $T_{\bullet}$ be a in Proposition \ref{prop:co-simplicial}. The co-simplicial ring $\Ab_{X,T}^{\bullet}$ injects into the co-simplicial ring 
$$[n] \mapsto \prod_{\sigma\colon [n] \rightarrow |X|_{\leq}} \Ab_{X,\sigma}.$$
\end{lemma}

We also need the following observation, which is a consequence of the definitions of ad\`eles.

\begin{rmk}\label{rmk:zero-dimensional}
If $\F$ is a quasi-coherent sheaf on $X$, set-theoretically supported on a finite union of closed points $Z \subset X$, then we have $\F  \simeq \A_X^k(\F)$.
\end{rmk}

Another observation which we will need, is that for an affine Noetherian scheme $X$, the  functor $$\Ab_{X,T}(-)\colon \QCoh(X) \rightarrow \Mod(\Gamma(\Oo_X))$$ can be expressed as $- \otimes_{\Gamma(\Oo_X)} \Ab_{X,T}$. This is the case, since $\Ab_{X,T}(-)$, and $- \otimes -$ commute with filtered colimits. Since $\Ab_{X,T}(-)$ is an exact functor, we see that $\Ab_{X,T}$ is a flat algebra over $\Gamma(\Oo_X)$. We record this for later use.

\begin{lemma}\label{lemma:adeles_flat}
Let $X = \Spec R$ be an affine Noetherian scheme, then $\Ab_{X,T}$ is a flat $R$-algebra.
\end{lemma}

\subsection{Functoriality}

If $f\colon X \rightarrow Y$ is a morphism of Noetherian schemes, we have an induced map of partially ordered sets $|X|_{\leq} \rightarrow |Y|_{\leq}$. Indeed, $x \in \overline{\{y\}}$ implies $f(x) \in \overline{\{f(y)\}}$. Additionally, we have an induced morphism of local rings $\Oo_{Y,f(x)} \rightarrow \Oo_{X,x}$. These observations are building blocks of a functoriality property satisfied by ad\`eles. To the best knowledge of the author, this property has not yet been recorded in the literature.

\begin{lemma}\label{lemma:functoriality}
Let $f\colon X \rightarrow Y$ be a morphism of Noetherian schemes, and $\F \in \QCoh(Y)$ a quasi-coherent sheaf on $Y$. For $T \subset |X|_n$, and $f(T) \subset T' \subset |Y|_n$ we have a morphism $\A_{Y,T'}(\F) \rightarrow f_*\A_{X,T}(f^*\F)$ in $\Mod(\Oo_Y)$, fitting into a commutative diagram
\[
\xymatrix{
\F \ar[r] \ar[d] & f_*f^*\F \ar[d] \\
\A_{Y,T'}(\F) \ar[r] & f_*\A_{X,T}(f^*\F).
}
\]
Moreover, this construction induces a map of augmented co-simplicial objects in $\Mod(\Oo_Y)$ from $\F \rightarrow \A_Y^{\bullet}(\F)$ to $f_*f^*\F \rightarrow f_*\A_X^{\bullet}(f^*\F)$.
\end{lemma}

\begin{proof}
The morphism $\A_{Y,T'}(\F) \rightarrow f_*\A_{X,T}(f^*\F)$ is constructed by induction on $n$ (where $T \subset |X|_n$). For $n = 0$ and $\F \in \Coh(Y)$, we have $\A_{X,T}(f^*\F) = \prod_{x \in T} \lim_{r \geq 0} (j_{rx})_*j_{rx}^*f^*\F$, and $\A_{Y,T'} \twoheadrightarrow \A_{Y,f(T)}(\F) = \prod_{x \in f(T)}  \lim_{r \geq 0} (j_{rx})_*j_{rx}^*\F$. We have a map $$\F \otimes_{\Oo_Y} \Oo_{Y,f(x)} \rightarrow f^*\F \otimes_{\Oo_X}\Oo_{X,x}$$ for each $x \in T$, which defines the required map for $T \subset |X|_0$.

Assume that the morphism $\A_{Y,f(T)}(\F) \rightarrow f_*\A_{X,T}(f^*\F)$ has been constructed for all $T \subset |X|_k$, where $k \leq n$. Let $T \subset |X|_{n+1}$. For every $x \in X$, we then have a well-defined map $\A_{Y,{}_{f(x)}f(T)}(\F) \rightarrow f_*\A_{X,{}_xT}(f^*\F)$, since $f({}_xT) \subset {}_{f(x)}f(T) \subset {}_{f(x)}T'$. This induces a morphism
$$ \prod_{x \in |X|}\lim_{r \geq 0} \A_{Y,{}_{f(x)}f(T)}((j_{rf(x)})_*j_{rf(x)}^*f^*\F) \rightarrow \prod_{x \in |X|} \lim_{r \geq 0} \A_{X,{}_xT}((j_{rx})_*j_{rx}^*f^*\F),$$
which we precompose with $$\prod_{y \in |Y|}\lim_{r \geq 0} \A_{Y,{}_yT'}((j_{ry})_*j_{ry}^*f^*\F) \twoheadrightarrow \prod_{x \in |X|}\lim_{r \geq 0} \A_{Y,{}_{f(x)}f(T)}((j_{rf(x)})_*j_{rf(x)}^*f^*\F)$$ to obtain the required morphism $\A_{Y,T'}(\F) \rightarrow f_*\A_{X,T}(f^*\F)$.
\end{proof}

Setting $\F = \Oo_Y$ in the assertion above, we obtain the following consequence.

\begin{corollary}
For a morphism of Noetherian schemes, we obtain a map of augmented co-simplicial objects 
\[
\xymatrix{
\Oo_Y \ar[r] \ar[d] & f_*\Oo_X \ar[d] \\
\A_Y^{\bullet}(\Oo_Y) \ar[r] & f_*\A_X(\Oo_X)
}
\]
in sheaves of algebras on the topological space $|Y|$.
\end{corollary}

Taking global sections, we obtain a functor from Noetherian schemes to co-simplicial rings.

\begin{definition}
We denote the functor $(\Sch^{\N})^{\op} \rightarrow (\Rings^{\Delta})$, sending a Noetherian scheme $X$ to the co-simplicial ring $\Ab_X^{\bullet}(\Oo_X)$ by $\Ab^{\bullet}$.
\end{definition}

As we have alluded to in Theorem \ref{thm:main4}, and will prove as Corollary \ref{cor:main4}, this functor has a left-inverse.

\subsection{Taking a closer look at the flasque sheaf of ad\`eles}

In this subsection we give a close analysis of flasqueness of the sheaf $\A_{X,T}(\F)$. We will show that the restriction map $\Ab_{X,T}(\F) \rightarrow \Ab_{U,T\cap |U|_n}(\F)$ is not only surjective, but admits an $\Ab_{X,T}(\Oo_X)$-linear section. As a consequence we obtain that $\A_{X,T}(\Oo_X)$ is what we call a \emph{l\^ache} sheaf of algebras in Definition \ref{defi:lache} (see also Corollary \ref{cor:adeles_lache}). A similar statement is contained in \cite[Proposition 2.1.5]{MR1138291}, however the $\A_{X,T}(\Oo_X)$-linearity is not investigated in \emph{loc. cit.}
 
\begin{lemma}\label{lemma:adeles_section}
Let $X$ be a Noetherian scheme, $T \subset |X|_n$, and $\F$ a quasi-coherent sheaf on $X$. For every open subset $U \subset X$ the restriction map $\Ab_{X,T}(\F) \rightarrow \Ab_U(T,\F)$ has a section, which is moreover $\Ab_{X,T}(\Oo_X)$-linear and functorial in $\F$.
\end{lemma} 
\begin{proof}
We denote the inclusion $U \hookrightarrow X$ by $j$, and will construct a section to the map of sheaves $\A_{X,T}(\F) \rightarrow j_*  \A_{U,T}(\F)$. Recall that for a coherent sheaf $\F$ we have an equivalence 
$\A_{X,T}(\F)  \simeq \prod_{x \in X} \varprojlim_r \A_{X,{}_xT}(j_{rx}^*\F),$ and $j_*\A_{U,T\cap|U|_n}(\F)  \simeq \prod_{x \in U} \varprojlim_r \A_{U,{}_xT\cap|U|_{n-1}}(j_{rx}^*\F).$
Suppose that we have already constructed a section $\phi_{\F}$ for $\A_{X,T'}(\F) \rightarrow j_*\A_{U,T'}(\F)$ for $T' \subset |X|_{n-1}$, such that for each map $\F \rightarrow \Gg$ we obtain a commutative square
\[
\xymatrix{
j_*\A_{U,T'}(\F) \ar[r]^{\phi_{\F}} \ar[d] & \A_{X,T'}(\F) \ar[d] \\
j_*\A_{U,T'}(\Gg) \ar[r]^{\phi_{\Gg}} & \A_{X,T'}(\Gg).
}
\]
We can then map the limit $\prod_{x \in U} \varprojlim_r \A_U({}_xT;j_{rx}^*\F)$ to $\prod_{x \in X} \varprojlim_r \A_U({}_xT;j_{rx}^*\F)$, by defining the components of the map to be $0$ for $x \in X\setminus U$, and given by the section $\phi$ otherwise.

By induction we see that we may assume that $T\subset |X|_0$. We may also assume that $\F$ is coherent. Hence, $\A_{X,T}(\F)$ is equal to the product $\prod_{x \in T} \varprojlim_r j_{rx}^*\F$, and $\A_U(T,\F)$ to $\prod_{x \in T\cap U} \varprojlim_r j_{rx}^*\F$. The natural restriction map is given by the canonical projection. A canonical section with the required properties is given by the identity map for components corresponding to $x \in U \cap T$, and the map $0$ for $x \in T \setminus U$.
\end{proof}

As a corollary one obtains the following observation of Beilinson.

\begin{corollary}\label{cor:adeles_flasque}
The sheaves $\Ab_{X,T}(\F)$ are flasque.
\end{corollary}

In \cite{MR565095} Beilinson continues to observe that for any quasi-coherent sheaf $\F$ the canonical augmentation $\F \rightarrow \A_X^{\bullet}(\F)$ induces an equivalence $\F  \simeq |\A_X^{\bullet}(\F)|$. A detailed proof is given by Huber \cite[Theorem 4.1.1]{MR1138291}. 

\begin{theorem}[Beilinson]\label{thm:beilinson}
Let $X$ be a Noetherian scheme and $\F$ a quasi-coherent sheaf on $X$. The augmentation $\F \rightarrow \A_X^{\bullet}(\F)$ defines a co-simplicial resolution of $\F$ by flasque $\Oo_X$-modules. Applying the global sections functor $\Gamma(X,-)$ we obtain $R\Gamma(X,\F) = |\Ab_X^{\bullet}(\F)|$, where the co-simplicial realisation $|\cdot|$ is taken in the derived $\infty$-category $\D(\AbGrp)$ of abelian groups.
\end{theorem}

It is instructive to test the general considerations above on the special case of algebraic curves. For the rest of this subsection we will thus assume that $X$ is an algebraic curve. We denote by $\A_X$ the sheaf, assigning to an open subset $U \subset X$ the ring of ad\`eles $\Ab_U$. Similarly we have sheaves $\mathsf{F}_X$, and $\mathsf{O}_X$, of rational functions, and integral ad\`eles.

The sheaves $\A_X$, and $\mathsf{O}_X$ satisfy the conclusion of Lemma \ref{lemma:adeles_section}, because a section over $U \subset X$ can be extended by $0$, outside of $U$. Since $\mathsf{F}_X(U) = \mathsf{F}_X(X)$, as long as $U \neq \emptyset$, we see that the conclusion of Lemma \ref{lemma:adeles_section} is trivially satisfied for $\mathsf{F}$. Beilinson's Theorem \ref{thm:beilinson} is in the present situation tantamount to the assertion that the complex 
$$[\Oo_X \rightarrow \mathsf{F}_X \oplus \mathsf{O}_X \rightarrow \A_X]$$
is exact. In other words, we observe that a rational function without any poles on $U \subset X$, defines a regular function on $U$. While this is a tautology in the $1$-dimensional case, the general setting of Noetherian schemes requires more subtle arguments from commutative algebra. We refer the reader to the proof of \cite[Theorem 4.1.1]{MR1138291} for more details.
 
\section{Perfect complexes and l\^ache sheaves of algebras}

In this section we introduce the notion of l\^ache sheaves of algebras, and prove Theorem \ref{thm:main3}.

\subsection{\emph{L\^ache} sheaves of algebras}

The main example of a l\^ache sheaf of algebras $\A$ is Beilinson's sheaf of ad\`eles. This is the content of Corollary \ref{cor:adeles_lache} below.

\subsubsection{Flasque sheaves} 

In this section we record a few well-known lemmas on flasque sheaves for the convenience of the reader.

\begin{lemma}\label{lemma:loc_flasque}
If $\F$ is a sheaf on $X$, such that every point $x \in X$ has an open neighbourhood $U$ with $\F|_U$ flasque, then $\F$ is a flasque sheaf.
\end{lemma}

\begin{proof}
Let $V \subset X$ be an open subset, and $s \in \F(V)$ a section. We claim that there exists $t \in \F(X)$ with $t|_V = s$. Consider the set $I$ of pairs $(W,t)$, where $W \subset X$ is an open subset, and $t \in \F(W)$, such that $t|_V = s$. Inclusion of open subsets induces a partial ordering on $I$, where we say that $(W,t) \leq (W',t')$ if $W \subset W'$, and $t'|_W = t$. Moreover, $I$ is inductively ordered, that is, for every totally ordered subposet $J \subset I$, there exists a common upper bound $i \in I$, such that we have $i \geq j$ for all $j \in J$. Indeed, denoting the pair corresponding to $j \in J$ by $(W_j,t_j)$, we have $W_j \subset W_k$ for $j \leq k$ in $J$, and $t_k|_{W_j} = t_j$. If we define $W = \bigcup_{j \in J} W_j$, the fact that $\F$ is a sheaf allows us to define a section $t \in \F(W)$ with $t|_{W_j} = t_j$. In particular, $(W,t) \in I$ is a common upper bound for the elements of $J$.

Zorn's lemma implies that the poset $I$ has a maximal element $(W,t)$. It remains to show that $W = X$. Assume that there exists $x \in X \setminus W$. By assumption, $x$ has an open neighbourhood $U$, such that $\F|_U$ is flasque. In particular, there exists a section $r \in \F(U)$, such that $r|_{U \cap W} = t|_{W \cap W}$. By virtue of the sheaf property we obtain a section $t' \in \F(W \cup U)$, satisfying $t'|_W = t$, which contradicts maximality of $(W,t)$.
\end{proof}

\begin{lemma}\label{lemma:loc_fg}
If $X$ is a quasi-compact topological space, and $\A$ a sheaf of algebras, then every locally finitely generated $\A$-module $\Mm$, which is flasque, is globally finitely generated, that is there exists a surjection $\A^n \rightarrow \Mm$.
\end{lemma}

\begin{proof}
For every point $x \in X$ there exists a neighbourhood $U_x$, such that $\Mm|_{U_x}$ is finitely generated. Since $X$ is quasi-compact, we may choose a finite subcover $X = \bigcup_{i = 1}^n U_i$, and generating sections $(s_{ij})_{j=1,\dots n_i}$. Because $\Mm$ is assumed to be flasque, we may extend each $s_{ij}$ to a global section $t_{ij}$, and see that this finite subset of $\Gamma(X,\Mm)$ generates $\Mm$.
\end{proof}

\begin{lemma}\label{lemma:resolution_flasque}
Assume that we have a short exact sequence of $\A$-modules with
$$0 \rightarrow \Mm_{2} \rightarrow \Mm_1 \rightarrow \Mm_0 \rightarrow 0,$$
with $\Mm_i$ flasque for $i > 0$, then $\Mm_0$ is flasque as well.
\end{lemma}

\begin{proof}
Since flasque sheaves have no higher cohomology, we have $H^1(X,\Mm_2) = 0$, and therefore the following commutative diagram has exact rows
\[
\xymatrix{
0 \ar[r] & \Mm_2(X) \ar[r] \ar[d] & \Mm_1(X) \ar[r] \ar[d] & \Mm_0(X) \ar[r] \ar[d] & 0 \\
0 \ar[r] & \Mm_2(U) \ar[r] & \Mm_1(U) \ar[r] & \Mm_0(U) \ar[r] & 0.
}
\]
Commutativity of the right hand square, and the fact that $\Mm_1(X) \twoheadrightarrow \Mm_1(U) \twoheadrightarrow \Mm_0(U)$ is surjective, implies surjectivity of $\Mm_0(X) \twoheadrightarrow \Mm_0(U)$.
\end{proof}

\begin{lemma}
Let $\A$ be an arbitrary sheaf of algebras on a topological space $X$. Consider the abelian category of sheaves of $\A$-modules. The full subcategory, given by $\A$-modules $\Mm$, such that $\Mm$ is a flasque sheaf, is extension-closed. 
\end{lemma}

\begin{proof}
Assume that we have a short exact sequence of $\A$-modules $\Mm_1 \hookrightarrow \Mm_2 \twoheadrightarrow \Mm_3$, with $\Mm_i$ flasque for $i = 1$ and $i = 3$. Since flasque sheaves are acyclic, we see that for every open subset $U \subset X$ we have a short exact sequence of abelian groups $\Mm_1(U) \hookrightarrow \Mm_2(U) \twoheadrightarrow \Mm_3(U)$. In particular, we obtain a commutative diagram with exact rows
\[
\xymatrix{
0 \ar[r] & \Mm_1(X) \ar[r] \ar@{->>}[d] & \Mm_2(X) \ar[r] \ar[d] & \Mm_3(X) \ar[r] \ar@{->>}[d] & 0 \\
0 \ar[r] & \Mm_1(U) \ar[r] & \Mm_2(U) \ar[r] & \Mm_3(U) \ar[r] & 0
}
\]
with the left and right vertical arrows being surjective. The Snake Lemma, or a simple diagram chase reveal that the vertical map in the middle also has to be surjective. This proves that $\Mm_2$ is a flasque sheaf.
\end{proof}

\begin{definition}
We denote the exact category, given by the extension-closed full subcategory of $\Mod(\A)$ consisting of modules whose underlying sheaf is flasque, by $\Mod_{\fl}(\A)$.
\end{definition}

We refer the reader to \cite{keller1996derived} and \cite{MR2606234} for the notion of derived categories of exact categories. We also emphasise that we use the notation $\D(-)$ to denote the stable $\infty$-category obtained by applying the dg-nerve construction of \cite[Section 1.3.1]{Lurie:ha} to the dg-category of \cite{keller1996derived}. The embedding $\Mod_{\fl}(\A) \hookrightarrow \Mod(\A)$ induces an exact functor of derived categories. 

It is important to emphasise that for a substantial part of this text we will not need to delve deeply into the theory of stable $\infty$-categories. The homotopy category of a stable $\infty$-category is naturally triangulated. In order to check that a functor $F \colon \C \rightarrow \D$ is fully faithful, essentially surjective, or an equivalence, it suffices to prove the same statement for its homotopy category (that is a classical triangulated category). This is essentially a consequence of the Whitehead Lemma. Distinguished triangles $X \rightarrow Y \rightarrow Z \rightarrow \Sigma X$ correspond to so-called \emph{bi-cartesian squares}
\[
\xymatrix{
X \ar[r] \ar[d] & Y \ar[d] \\
0 \ar[r] & Z,
}
\]
that is, commutative diagrams, which are cartesian and co-cartesian.

\begin{lemma}\label{lemma:c2}
The canonical functor 
$\D^+(\Mod_{\fl}(\A)) \rightarrow \D^+(\A)$, induced by the exact functor $\Mod_{\fl}(\A) \hookrightarrow \Mod(\A)$, is fully faithful.
\end{lemma}

\begin{proof}
According to a theorem of Keller \cite[Theorem 12.1]{keller1996derived} it suffices to check that every short exact sequence of $\A$-modules $\Mm_1 \hookrightarrow \Mm_2 \twoheadrightarrow \Mm_3$ with $\Mm_1$ flasque, fits into a commutative diagram with exact rows
\[
\xymatrix{
0 \ar[r] & \Mm_1 \ar[r] \ar@{=}[d] & \Mm_2 \ar[r] \ar[d] & \Mm_3 \ar[r] \ar[d] & 0 \\
0 \ar[r] & \Mm_1 \ar[r] & \Mm_2' \ar[r] & \Mm_3' \ar[r] & 0
}
\]
with $\Mm_2'$ and $\Mm_3'$ flasque. In order to produce this diagram, we recall that every $\A$-module $\Mm$ can be embedded into a flasque $\A$-module. Indeed, the sheaf of discontinuous sections, that is, $\Mm^{dc}(U) = \prod_{x \in U} \Mm_x$ provides such an embedding. Let $\Mm_2 \hookrightarrow \Mm_2'$ be an embedding of $\Mm_2$ into a flasque $\A$-module. Then, the quotient $\Mm_3'=\Mm_2'/\Mm_1$ is also flasque.
\end{proof} 

\subsubsection{Flasque sheaves of algebras}

In this paragraph we will ponder over what can be said about quasi-coherent sheaves of $\A$-modules, if the sheaves of algebras $\A$ itself is known to be flasque. Recall that an $\A$-module $\Mm$ is \emph{quasi-coherent}, if every point $x \in X$ has a neighbourhood $U \subset X$, such that the restriction $\Mm|_U$ can be represented as a cokernel of a morphism $\A^{\oplus J}|_U \rightarrow \A^{\oplus I}|_U$ of free $\A$-modules.

\begin{rmk}\label{rmk:qcoh}
For a general sheaf of algebras $\A$ the category $\QCoh(\A)$ of quasi-coherent $\A$-modules is in general not closed under taking kernels in the abelian category of $\A$-modules $\Mod(\A)$. In particular, one does not expect $\QCoh(\A)$ to be abelian in general. If the restriction maps $\A(V) \rightarrow \A(U)$ for $U \subset V$ belonging to a specific subbase for the topology, are known to be flat, $\QCoh(\A)$ can be shown to be abelian. This assumption is too strong for the sheaves of algebras we care about in this article.
\end{rmk}

We see from Lemma \ref{lemma:loc_flasque} that every locally free or locally projective sheaf of $\A$-modules is flasque. In general one cannot expect every quasi-coherent sheaf of $\A$-modules to be flasque. However, we will see in the next paragraph that there are certain flasque sheaves of algebras, for which this is true.

\begin{lemma}\label{lemma:c2_projective}
Let $\A$ be a sheaf of algebras on $X$, such that every free $\A$-module is flasque. We denote by $\P(\A)$ the exact category given by the idempotent completion of free $\A$-modules, and refer to its objects projective $\A$-modules. The functor $\D^-(\P(\A)) \hookrightarrow \D^-(\Mod_{\fl}(\A))$, induced by the inclusion $\P(\A) \hookrightarrow \Mod_{\fl}(\A)$, is fully faithful.
\end{lemma}

\begin{proof}
We will apply the dual of the result in Keller \cite[Theorem 12.1]{keller1996derived}, by which it suffices to check that every short exact sequence of flasque $\A$-modules $\Mm_1 \hookrightarrow \Mm_2 \twoheadrightarrow \Mm_3$ with $\Mm_3$ projective, fits into a commutative diagram with exact rows
\[
\xymatrix{
0 \ar[r] & \Mm'_1 \ar[r] \ar[d] & \Mm'_2 \ar[r] \ar[d] & \Mm_3 \ar[r] \ar@{=}[d] & 0 \\
0 \ar[r] & \Mm_1 \ar[r] & \Mm_2 \ar[r] & \Mm_3 \ar[r] & 0
}
\]
with $\Mm'_1$ and $\Mm'_2$ projective. Since $\Mm_2 \in \Mod_{\fl}(\A)$ is flasque by assumption, there exists a surjection $\A^{\oplus I} \rightarrow \Mm_2$ for some index set $I$. Indeed we can take $I = \{(U,s)|U\subset X \text{ open, }s\in\Mm_2(U)\}$. Choosing an extension $t_{(U,s)} \in \Mm_2(X)$, satisfying $t_{(U,s)}|_U = s$ for every element of $I$, we obtain a surjective morphism of $\A$-module $\A^{\oplus I} \rightarrow \Mm_2$.

Let $\Mm'_2 = \A^{\oplus I}$, and define $\M'_1$ to be the kernel of the composition $\Mm'_2 \rightarrow \Mm_2 \rightarrow \Mm_3$. Since $\Mm_3$ is a direct summand of a free $\A$-module, and $\Mm_2'$ is flasque, there exists a splitting to this surjection. Therefore $\Mm'_1$ belongs to $\P(\A)$, since it is a direct summand of $\Mm'_2$. This concludes the proof.
\end{proof}

\subsubsection{Definition of l\^ache sheaves of algebras} 

A sheaf $\A$ on a topological space $X$ is called \emph{flasque}, if for every open subset $U \subset X$ the restriction map $\A(X) \rightarrow \A(U)$ is surjective. If $\A$ carries additionally the structure of a sheaf of algebras, there is a strengthening of this condition.
\begin{definition}\label{defi:lache}
A sheaf of algebras $\A$ on $X$ is called \emph{l\^ache}, if for every open subset $U \subset X$, and every map of free $\A_U$-modules $\A^{\oplus J}_U \xrightarrow{f} \A^{\oplus I}_U$ the kernel $\ker f$ is a flasque sheaf on $U$.
\end{definition}

To see that there are non-trivial l\^ache sheaves of algebras, we let $X$ be a topological space where every open subset is also closed, in the following example.

\begin{example}
Let $X$ be a topological space, where every open subset is also closed. Then every sheaf of abelian groups $\F$ is flasque. If $U \subset X$ is open, and $s \in \F(U)$, then using the sheaf property of $\F$ we see that there is a unique section $t \in \F(X)$, such that $t|_U = s$ and $t|_{X \setminus U} = 0$. This is possible because $X \setminus U$ is open by assumption. Hence, every sheaf of algebras on $X$ is l\^ache.
\end{example}

The Lemma below implies that for a l\^ache sheaf of algebras $\A$, and a morphism $f \colon \A^{\oplus J} \rightarrow \A^{\oplus I}$ the sheaves $\image f$ and $\coker f$ are flasque as well.

\begin{lemma}\label{lemma:123}
Let $V_1 \xrightarrow{f} V_2$ be a morphism of flasque sheaves, such that $\ker f$ is flasque. Then, the sheaves $\image f$, and $\coker f$ are flasque.\end{lemma}
\begin{proof}
We have a short exact sequence $\ker f \hookrightarrow V_1 \twoheadrightarrow \image f$, since the first two sheaves are flasque, so is the third (Lemma \ref{lemma:resolution_flasque}). The same argument applies to the short exact sequence $\image f \hookrightarrow V_2 \twoheadrightarrow \coker f$, and implies that $\coker f$ is flasque.
\end{proof}
We can further generalise the assertion.
\begin{lemma}\label{lemma:lache_projective}
Let $V_1 \xrightarrow{f} V_2$ be a morphism of projective quasi-coherent $\A$-modules (that is, direct summands of free modules), where $\A$ is l\^ache. Then the sheaves $\ker f$, $\image f$, and $\coker f$ are flasque.
\end{lemma}
\begin{proof}
Since every projective quasi-coherent $\A$-module is a direct summand of a free $\A$-module, there exist quasi-coherent $\A$-modules $W_1$, and $W_2$, such that $V_i \oplus W_i$ are free $\A$-modules for $i = 1,2$. The induced map
$$f \oplus \id \colon V_1 \oplus (W_1 \oplus W_2 \oplus V_1 \oplus V_2)^{\oplus \mathbb{N}} \rightarrow V_2 \oplus (W_1 \oplus W_2 \oplus V_1 \oplus V_2)^{\oplus \mathbb{N}}$$
has the same kernel $\ker f  \simeq \ker (f \oplus \id)$. However, the Eilenberg swindle
$$(W_1 \oplus W_2 \oplus V_1 \oplus V_2)^{\oplus \mathbb{N}}  \simeq (W_1)^i \oplus (W_2)^j \oplus (W_1 \oplus W_2 \oplus V_1 \oplus V_2)^{\oplus \mathbb{N}}$$
allows us to see that the two sides are in fact free $\A$-modules. Therefore, the defining property of l\^ache sheaf of algebras implies that $\ker f$ is flasque. Lemma \ref{lemma:123} yields that $\image f$ and $\coker f$ are flasque sheaves.
\end{proof}

The considerations above imply in particular that every quasi-coherent $\A$-module of a l\^ache sheaf of algebras $\A$ is flasque. However, we have to keep in mind that the category of quasi-coherent sheaves is not abelian in general, as we pointed out in Remark \ref{rmk:qcoh}.
We have the following corollary to Lemma \ref{lemma:lache_projective}.
\begin{corollary}\label{cor:cohomology_flasque}
If $\Mm^{\bullet}$ is locally equivalent to an object of $\D(\P(\A))$, then its cohomology sheaves $\Hc^i(\Mm^{\bullet})$ are flasque.
\end{corollary}

\begin{proof}
We have seen in Lemma \ref{lemma:loc_flasque} that a sheaf is flasque if and only if it is locally flasque. Therefore we may assume $M \in \D(\P(\A))$. Let us choose an explicit presentation by a complex $(V^{\bullet},d)$, where each $V^i$ is a projective $\A$-module. We have $\Hc^i(M)  \simeq \frac{\ker d^i}{\image d^{i-1}}$. By Lemma \ref{lemma:lache_projective}, $\ker d^i$, and $\image d^{i-1}$ are flasque. By Lemma \ref{lemma:resolution_flasque}, the quotient $\Hc^i(M)$ is flasque.
\end{proof}

\subsubsection{A criterion for being l\^ache}

In this paragraph we observe that every sheaf of algebras $\A$, which admits linear sections to the restriction maps $\A(X) \rightarrow \A(U)$, is in fact l\^ache. As a consequence, we obtain that the sheaf of ad\`eles on a Noetherian scheme is l\^ache (Corollary \ref{cor:adeles_lache}).

\begin{definition}
A sheaf of algebras $\A$ is called \emph{very flasque}, if for every open subset $U \subset X$ there exists an $\A(X)$-linear section $\phi_U$ of the restriction map $r_U \colon \A(X) \rightarrow \A(U)$.
\end{definition}

Typically the section $\phi_U$ is given by a map which extends $s \in \A(U)$ by $0$ outside of $U$, as in the following example.

\begin{example}
Let $X$ be a topological space where every open subset is closed, and $\A$ an arbitrary sheaf of algebras, then $\A$ is very flasque.
\end{example}
\begin{proof}
Since $X$ is totally disconnected, every open set $U \subset X$ is also closed. Thus $X\setminus U$ is also open, and we may define $\phi\colon \A(U) \rightarrow \A(X)$ to be the map which sends $s \in \A(U)$ to the unique section $\widehat{s} \in \A(X)$, such that $\widehat{s}|_U = s$, and $\widehat{s}|_{X\setminus U} = 0$. This map is indeed $\A(X)$-linear.
\end{proof}

In hindsight we have shown in Lemma \ref{lemma:adeles_section} that for every quasi-coherent sheaf of algebras $\F$ on a Noetherian scheme $X$, the sheaves of algebras $\Ab_{X,T}(\F)$ are very flasque. See also Corollary \ref{cor:adeles_lache} below, where an important consequence of this observation is recorded.

The next lemma is the aforementioned criterion for a sheaf of algebras being l\^ache.

\begin{lemma}\label{lemma:lache}
A very flasque sheaf of algebras $\A$ is l\^ache.
\end{lemma}
\begin{proof}
Let $f\colon \A^{\oplus J}_V \rightarrow \A^{\oplus I}_V$ be an $\A_V$-linear map, where $V \subset X$ is open. We have to show that $K=\ker f$ is a flasque sheaf on $V$. For $U \subset V$ open we have a commutative diagram
\[
\xymatrix{
0 \ar[r] & K(V) \ar[r] \ar[d] & \A^{\oplus J}(V) \ar[r] \ar[d] & \A^{\oplus I}(V) \ar[d] \\
0 \ar[r] & K(U) \ar[r] & \A^{\oplus J}(U) \ar[r] & \A^{\oplus I}(U)
}
\]
with exact rows, because taking global sections is a left exact functor. However, $\A(V)$-linearity of the section $r_V \circ \phi_U\colon \A(U) \rightarrow \A(V)$ implies that we have a commutative diagram
\[
\xymatrix{
0 \ar[r] & K(V) \ar[r]  & \A^{\oplus J}(V) \ar[r]  & \A^{\oplus I}(V) ] \\
0 \ar[r] & K(U) \ar[r] \ar@{-->}[u] & \A^{\oplus J}(U) \ar[r] \ar[u] & \A^{\oplus I}(U) \ar[u] 
}
\]
where the dashed arrow is provided by the universal property of kernels. The dashed arrow is therefore right-inverse to the restriction map $K(V) \rightarrow K(U)$, and we conclude that $K = \ker f$ is flasque.
\end{proof}

\begin{corollary}\label{cor:adeles_lache}
For a Noetherian scheme $X$ and a quasi-coherent sheaf $\F$ of algebras, the sheaves of Beilinson-Parshin ad\`eles $\Ab_{X,T}(\F)$ are l\^ache sheaves of algebras.
\end{corollary}
\begin{proof}
Lemma \ref{lemma:adeles_section} asserts that $\Ab_{X,T}(\F)$ is very flasque. According to Lemma \ref{lemma:lache} this implies that $\Ab_{X,T}(\F)$ is also l\^ache.
\end{proof}

\subsection{Perfect complexes}
In this subsection we study the $\infty$-category of perfect complexes of $\A$-modules. This is necessary since the classical category of quasi-coherent $\A$-modules is not necessarily abelian (see Remark \ref{rmk:qcoh}).
\begin{definition}\label{defi:Dminus}
Let $\P(\A)$ denote the exact category obtained as the idempotent completion of the exact category of free $\A$-modules. We denote by $\D^{-}(\A)$ the $\infty$-category corresponding to the full subcategory of $\D(\Mod_{\fl}(\A))$ given by complexes of flasque $\A$-modules, which are locally equivalent to objects of $\D^-(\P(\A))$. 
\end{definition}

Recall that every exact functor between exact categories, induces a functor between derived $\infty$-categories.
\begin{lemma}\label{lemma:Gamma_conservative}
Let $X$ be a quasi-compact topological space, and $\A$ a l\^ache sheaf of algebras on $X$. We denote by $R = \Gamma(\A)$ the ring of global sections of $\A$. The global sections functor $\Gamma\colon \D^-(\A) \rightarrow \D^-(R)$, induced by the exact functor $\Gamma\colon \Mod_{\fl}(\A) \rightarrow \Mod(R)$, is well-defined and conservative.
\end{lemma}
\begin{proof}
Pick a complex $\Mm^{\bullet} \in \D(\Mod_{\fl}(\A))$ representing an object of $\D^-(\A)$. By definition we have $\Gamma(\Mm^{\bullet}) = (\Gamma(\Mm)^{\bullet})$. Moreover, since the cohomology sheaves $\Hc^i(\Mm^{\bullet})$ are flasque (Corollary \ref{cor:cohomology_flasque}), we have that $\Hc^i(\Gamma(\Mm^{\bullet}))  \simeq \Gamma(\Hc^i(\Mm^{\bullet}))$. In particular we see that $\Gamma(\Mm^{\bullet})$ is an object of $D^-(R)$ (since $X$ is quasi-compact). If $\Mm \neq 0$, there exists $\Hc^i(\Mm^{\bullet}) \neq 0$. Thus, since it is a flasque sheaf, its global sections must be non-zero. We conclude that $\Gamma$ is indeed a conservative functor, as asserted.
\end{proof}

We also have a localisation functor.

\begin{definition}\label{defi:Loc}
For $\A$ a l\^ache sheaf of algebras on $X$, we have an exact functor between exact categories $- \otimes_R \A \colon \P(R) \rightarrow \P(\A) \hookrightarrow \Mod_{\fl}(\A)$. The induced exact functor between derived $\infty$-categories will be denoted by $$\Loc\colon \D^-(R) \rightarrow \D^-(\A).$$
\end{definition}
\begin{proposition}
If $X$ is quasi-compact and $\A$ is l\^ache, then $\Gamma\colon \D^-(\A) \rightarrow \D^-(R)$ is an equivalence of $\infty$-categories, with inverse equivalence $\Loc$.
\end{proposition}
\begin{proof}
There is a commutative triangle
\[
\xymatrix{
\P(R) \ar[r]^-{-\otimes_R \A} \ar[rd]_{\id} & \Mod_{\fl}(\A) \ar[d]^{\Gamma} \\
& \P(R) 
}
\]
of exact functors, inducing a natural equivalence of functors $\id_{\D^-(\A)} \simeq \Loc \circ \Gamma$. We claim that we also have an equivalence $\Gamma \circ \Loc \simeq \id_{\D^-(R)}$. To see this, consider $\Mm^{\bullet} \in \D^-(\A)$. We will show that $\Mm^{\bullet}$ belongs to the essential image of $\Loc$. Let $g\colon P^{\bullet} \rightarrow \Gamma(\Mm^{\bullet})$ be a projective replacement of $\Gamma(\Mm^{\bullet})$, given by an actual morphism between chain complexes in $\Mod(R)$. By the adjunction between $- \otimes_R \A$ and $\Gamma$, this yields a morphism 
$f\colon \Loc(P^{\bullet}) \simeq \Loc(\Gamma(\Mm^{\bullet})) \rightarrow \Mm^{\bullet}$ in $\D^-(\A)$. Since $\Gamma(f) = g$ is a quasi-isomorphism, and $\Gamma$ is conservative by Lemma \ref{lemma:Gamma_conservative}, we see that $f$ is an equivalence. This implies that every $\Mm^{\bullet} \in \D^-(\A)$ is in fact equivalent to an object of $\D^-(\P(\A))$. Therefore we have a natural equivalence $\Loc \circ \Gamma \simeq \id_{\D^-(\A)}$ as a consequence of the commutative diagram
\[
\xymatrix{
\P(\A) \ar[r]^-{\Gamma} \ar[rd]_{\id} & \P(R) \ar[d]^{-\otimes_R \A} \\
& \P(\A) 
}
\]
of exact functors, and Lemma \ref{lemma:c2_projective}, which asserted $\D^-(\P(\A)) \hookrightarrow \D^-(\Mod_{\fl}(\A))$.
\end{proof}

\begin{corollary}\label{cor:pre_eq_lache}
The functor $\Gamma\colon \D^-(\A) \rightarrow \D^-(R)$ can be promoted to a symmetric monoidal equivalence of symmetric monoidal $\infty$-categories.
\end{corollary}
\begin{proof}
The functor $\Loc$ lifts to a symmetric monoidal functor, since $-\otimes_R \A$ is a symmetric monoidal functor $P(R) \to P(\A)$. Therefore also the inverse functor $\Gamma$ can be canonically lifted to a symmetric monoidal functor.
\end{proof}

\begin{corollary}\label{cor:eq_lache}
The functor $\Gamma$ restricts to a symmetric monoidal equivalence of symmetric monoidal stable $\infty$-categories
$$\Perf(\A)_{\otimes} \rightarrow \Perf(R)_{\otimes},$$
with inverse $\Loc$.
\end{corollary}
\begin{proof}
This follows from the equivalence $\D^-(\A)_{\otimes}  \simeq \D^-(R)_{\otimes}$ once we have shown that the mutually inverse functors $\Gamma$ and $\Loc$ induce functors $\Gamma \colon \Perf(\A) \rightarrow \Perf(R)$, and $\Loc\colon \Perf(R) \rightarrow \Perf(A)$. Since $\Loc$ is induced by $- \otimes_R A$ it is clear that $\Loc$ sends perfect complexes to perfect complexes. 

In order to show that $\Gamma$ preserves perfect complexes, one uses that they are dualisible in $\D^-(\A)_{\otimes}$. The details are as follows. We have shown in Lemma \ref{lemma:c2} that the functor $\D^+(\Mod_{\fl}(\A)) \hookrightarrow \DMod^+(\A)$ is fully faithful. In particular, let $\Mm^{\bullet} \in \Perf(\A) \subset \DMod(\A)$ be a chain complex, which is locally equivalent to bounded complexes of projective $\A$-modules. Since $X$ is quasi-compact, every perfect complex is automatically bounded, and therefore admits a bounded below resolution by flasque $\A$-modules. Therefore, $\Mm^{\bullet} \in \D^-(\A)$. Moreover, $(\Mm^{\bullet})^{\vee}$ also belongs to $\D^-(\A)$, and hence $\Mm^{\bullet}$ is dualisable in $\D^-(\A)$. We conclude that $\Gamma(\Mm^{\bullet})$ is a dualisable object in $\D^-(R)$, that is, $\Gamma(\Mm^{\bullet}) \in \Perf(R)$. Since $\Gamma$ and $\Loc$ are symmetric monoidal functors, they preserve perfect complexes. 
\end{proof}

\section{Adelic descent for perfect complexes and maps to stacks}

This section is concerned with proving the Adelic Descent Theorem, and exploring its consequences. We refer the reader to Definition \ref{defi:adeles} for Beilinson's sheaf of ad\`eles.

\begin{theorem}[Adelic Descent]\label{thm:adelic_descent}
For every Noetherian scheme $X$ we have an equivalence of symmetric monoidal $\infty$-categories
$$\Perf(X)_{\otimes}  \simeq |\Perf(\Ab_X^{\bullet})_{\otimes}|,$$
where $\Ab_X^{\bullet}$ denotes the co-simplicial ring of Beilinson-Parshin ad\`eles, and $|\Perf(\Ab_X^{\bullet})_{\otimes}|$ denotes the totalisation of the co-simplicial object $\Perf(\Ab_X^{\bullet})_{\otimes}$ in the $\infty$-category of small monoidal stable $\infty$-categories.
\end{theorem}

Let $S^{\bullet}$ be a co-simplicial ring. Applying the functor $\Perf\colon \Rings \rightarrow \inftyCat$, assigning to a ring $R$ the stable $\infty$-category of perfect $R$-complexes, we obtain a co-simplicial object in stable $\infty$-categories, $\Perf(S^{\bullet})$. The objects of the limit $\infty$-category $|\Perf(S^{\bullet})|$ will also be referred to as \emph{cartesian} co-simplicial perfect $S^{\bullet}$-complexes. This formulation is inspired by a co-simplicial formulation of Grothendieck's faithfully flat descent theory. We will briefly review this viewpoint, in order to motivate the following considerations. More details can be found in \cite[Tag 039Y]{stacks-project}. For $R \rightarrow S$ a ring homomorphism, we may define a co-simplicial ring $S^{\bullet}$, where $S^k$ can be identified with tensoring $S$ with itself $(k+1)$-times; $S^k = S \otimes_R \cdots \otimes_R S$. A co-simplicial $S^{\bullet}$-module consists of a co-simplicial $R$-module $M^{\bullet}$, such that $M^k$ is an $S^k$-module, compatibly with the co-simplicial structure on $M^{\bullet}$ and $S^{\bullet}$. We say that $M^{\bullet}$ is a cartesian $S^{\bullet}$-module, if for every $[k] \rightarrow [n]$ the induced map $M^k \otimes_{R^k} R^n \rightarrow M^n$ is an isomorphism of $R^n$-modules. Classical descent theory asserts that for $R \rightarrow S$ faithfully flat, we have an equivalence of categories between $R$-modules and cartesian $S^{\bullet}$-modules.

\subsection{Proof of the Adelic Descent Theorem}

The Adelic Descent Theorem \ref{thm:adelic_descent} follows from an analogous assertion for sheaves of symmetric monoidal $\infty$-categories.

\begin{proposition}\label{prop:adelic_descent}
For every Noetherian scheme $X$ we have an equivalence of sheaves of symmetric monoidal $\infty$-categories
$$- \otimes \A_X^{\bullet} \colon \Perf(\Oo_X)_{\otimes}  \simeq |\Perf(\A_X^{\bullet})_{\otimes}|,$$
where $\A_X^{\bullet}$ denotes the co-simplicial sheaf of algebras of Beilinson-Parshin ad\`eles.
\end{proposition}

In order to deduce the aforementioned theorem from this result, one takes global section: we have seen in Corollary \ref{cor:eq_lache} that taking global sections is a symmetric monoidal equivalence. More details can be found in Paragraph \ref{sub:conclusion}.

\subsubsection{The adelic realisation functor}

Our proof requires an $\infty$-category ambient to $|\Perf(\A_X^{\bullet})|$, closed under small colimits. This is necessary to apply Lurie's Adjoint Functor Theorem \ref{thm:adjointfunctor}.

\begin{definition}\label{adelic_realisation}
For a sheaf of algebras $\A$ on a topological space $X$, we define $\QC(\A) = \Ind \Perf(\A)$.
Let $X$ be a Noetherian scheme, we define the category $|\QC(\A_X^{\bullet})|$ to be the geometric realisation of the co-simplicial object $\QC(\A_X^{\bullet})$ in symmetric monoidal $\infty$-categories.
\end{definition}

Below we define the adelic realisation functor.

\begin{definition}\label{adelic_realisation}
The \emph{adelic realisation functor} $\Perf(\Oo_X)_{\otimes} \rightarrow \Perf(\A_X^{\bullet})_{\otimes}$ is given by tensoring a perfect complex of $\Oo_X$-modules with the co-simplicial sheaf of rings $\A_X^{\bullet}$. Since $\QC(\Oo_X)_{\otimes} \rightarrow \Ind \Perf(\Oo_X)_{\otimes}$, and analogously $\QC(\A_X^{k})  \simeq \Ind \Perf(\A_X^{k})$ for every $[k] \in \Delta$ we obtain an extension of this functor
$$- \otimes \A_X^{\bullet} \colon \QC(\Oo_X)_{\otimes} \rightarrow \QC(\A_X^{\bullet})_{\otimes}.$$
\end{definition}

In light of Proposition \ref{prop:adelic_descent} it is tempting to believe that the adelic descent functor $\QC(X) \rightarrow |\QC(\A_X^{\bullet})|$ is an equivalence. However, due to pathological behaviour of complexes of modules over infinite products of rings, we do not expect this to be true.
The $\infty$-category $|\QC(\A_X^{\bullet})|$ is stable, and admits small colimits for formal reasons.

\begin{lemma}
The $\infty$-category $|\QC(\A_X^{\bullet})|$ is stable, possesses a canonical symmetric monoidal structure, and is closed under small colimits.
\end{lemma}

\begin{proof}
By definition, $|\QC(\A_X^{\bullet})|$ is the limit of a co-simplicial diagram of symmetric monoidal stable $\infty$-categories. Therefore it is stable (see \cite[Proposition 1.1.4.6]{Lurie:ha}) and symmetric monoidal itself.

By Proposition \ref{prop:colimits}, it suffices to verify that $|\QC(\A_X^{\bullet})|$ is closed under small coproducts. Let $I$ be a set, and $(M_i^{\bullet})_{i \in I} \in |\QC(\A_X^{\bullet})|$ a family of objects in $|\QC(\A_X^{\bullet})|$. Since the functor $\otimes$ commutes factorwise with small coproducts, we see that $\bigoplus_{i \in I} M_i^{\bullet} \in |\QC(\A_X^{\bullet})|$, and satisfies the universal property of a coproduct.
\end{proof}

We also know that this $\infty$-category is locally small by Remark \ref{rmk:loc_small}.

\subsubsection{The right adjoint to adelic realisation}

The adelic realisation functor preserves small coproducts. This follows from its definition using tensor products, and the description of mapping spaces in the $\infty$-category $|\QC(\A_X^{\bullet})|$ given in Lemma \ref{lemma:limits}. Since it is exact, it commutes with small colimits by virtue of Proposition \ref{prop:colimits}. Moreover, we know that $\QC(X) \cong \Ind \Perf(X)$ is presentable. By Lurie's Adjoint Functor Theorem \ref{thm:adjointfunctor} we see that there exists a right adjoint $\int_X \colon |\QC(\A_X^{\bullet})| \rightarrow \QC(X)$. 
\begin{definition}
The right adjoint to the functor $- \otimes_{\Oo_X} \A_X^{\bullet}\colon \QC(X) \rightarrow |\QC(\A_X^{\bullet})|$ will be denoted by $\int_X\colon |\QC(\A_X^{\bullet})| \rightarrow \QC(X)$.
\end{definition}

It is possible to give a more concrete definition of the functor $\int_X$. For a discussion in terms of model categories we refer the reader to To\"en--Vezzosi's treatment \cite[1.2.12]{MR2394633}, which also inspired the notation $\int$. We will content ourselves with the following observation: by virtue of the equivalence $\QC(X)  \simeq \Ind \Perf(X)$, we may view objects in $\QC(X)$ as functors $\Perf(X)^{\op} \rightarrow \Grpd$, sending finite colimits in $\Perf(X)$ (that is, finite limits in $\Perf(X)^{\op}$) to finite limits in $\Grpd$. The adjunction between $-\otimes_{\Oo_X} \A_X^{\bullet}$ and $\int_X$ implies that for $N^{\bullet} \in \QC(\A_X^{\bullet})$, $\int_X N^{\bullet}$ is the functor $\Perf(X)^{\op} \rightarrow \Grpd$, which informally is given by
$$M \mapsto \Map_{\A_X^{\bullet}}(M \otimes_{\Oo_X}\A_X^{\bullet},N^{\bullet}).$$
It follows from the proof of the Adjoint Functor Theorem that $\int_X$ is equivalent to the functor 
$$\Map_{\A_X^{\bullet}}(-\otimes_{\Oo_X} \A_X^{\bullet},-)\colon \QC(\A_X) \times \Perf(X)^{\op} \rightarrow \Grpd.$$
We record this for later use.

\begin{lemma}\label{lemma:intHom}
With respect to the equivalence $\QC(X)  \simeq \Ind \Perf(X) \subset \Fun(\Perf(X)^{\op},\Grpd)$, we have $\int_X \simeq \Map_{\A_X^{\bullet}}(-\otimes_{\Oo_X}\A_X^{\bullet},-)\colon \QC(\A_X^{\bullet}) \times \Perf(X)^{\op} \rightarrow \Grpd$. 
\end{lemma}

Another formal property of $\int_X$ of importance to us is its behaviour with respect to small colimits. We will give a heuristic justification that $\int_X$ commutes with small colimits, before engaging with the formal argument. Since $|X|$ is a Noetherian topological space (that is every open subset is quasi-compact), we have $\Perf(\A) \subset \DMod(\A)^c$, that is, every perfect complex of $\A$-modules is compact. Thinking informally of $\int_X$ as a functor $\QC(\A_X^{\bullet}) \times \Perf(X)^{\op} \rightarrow \Grpd$, we see that 
$$\int_X \bigoplus_{i \in I} N_i^{\bullet} \simeq \bigoplus_{i \in I}\int_X N_i^{\bullet}\colon M \mapsto \Hom_{\A_X^{\bullet}}(M \otimes_{\Oo_X} \A_X^{\bullet},\bigoplus_{i \in I} N_i^{\bullet})  \simeq \bigoplus_{i \in I}\Hom_{\A_X^{\bullet}}(M \otimes_{\Oo_X} \A_X^{\bullet},N_i^{\bullet}).$$
The key observation of the heuristic reasoning above is that $-\otimes_{\Oo_X} \A_X^{\bullet}$ preserves compact objects. Therefore we need a sufficient criterion for objects in $\QC(\A_X^{\bullet})$ to be compact.

\begin{lemma}\label{lemma:perf_compact}
Let $M^{\bullet} \in \QC(\A_X^{\bullet})$ be a cartesian co-simplicial $\A_X^{\bullet}$-module. If $M^0 \in \Perf(\A_X^0) \subset \QC(\A_X^0)$, then $M^{\bullet}$ is compact in $\QC(\A_X^{\bullet})$.
\end{lemma}

\begin{proof}
Let $(N_i^{\bullet})_{i \in I}$ be an arbitrary small family of objects in $\QC(\A_X^{\bullet})$, and $f\colon M^{\bullet} \rightarrow \bigoplus_{i \in I} N_i^{\bullet}$ an arbitrary morphism. It suffices to show that there exists a finite subset $J \subset I$, such that we have a factorisation
\begin{equation}\label{eqn:factorisation}
\xymatrix{
M^{\bullet} \ar[r] \ar[rd] & \bigoplus_{i \in J} N_i^{\bullet} \ar[d] \\
& \bigoplus_{i \in I} N_i^{\bullet}.
}
\end{equation}
In order to produce this factorisation, we will show that the space of commutative diagrams as drawn above is either empty or contractible. Since $\Perf(\A_X^0) \subset \QC(\A_X^0)^c$, by virtue of the definition $\QC(\A_X^0)  \simeq \Ind \Perf(\A_X^0)$, we see that we have a finite subset $J \subset I$ and a factorisation $M^0 \rightarrow \bigoplus_{i \in J} N_i^{0} \rightarrow \bigoplus_{i \in I} N_i^{0}$. Since the modules $M^{\bullet}$ and $(N_i^{\bullet})_{i \in I}$ are cartesian, that is, $M^k  \simeq M^0 \otimes_{\A_X^0} \A_X^k$, we see that we have a factorisation  $M^k \rightarrow \bigoplus_{i \in J} N_i^{k} \rightarrow \bigoplus_{i \in I} N_i^{k}$ for all co-simplicial levels $k \geq 0$. 

Consider the cartesian diagram of co-simplicial spaces
\[
\xymatrix{
X^{\bullet} \ar[r] \ar[d] & \mathsf{pt}^{\bullet} \ar[d]^f\\
\Map_{\A_X^{\bullet}}(M,\bigoplus_{i \in J} N_i^{\bullet}) \ar[r] & \Map_{\A_X^{\bullet}}(M,\bigoplus_{i \in I} N_i^{\bullet}),
}
\]
where $\mathsf{pt}^{\bullet}$ is the constant diagram, consisting levelwise of a single point, and the map $\mathsf{pt}^{\bullet} \rightarrow \Map(M,\bigoplus_{i \in I} N_i^{\bullet})$ corresponds to the element $f \in |\Map(M,\bigoplus_{i \in I} N_i^{\bullet})|$ (where we used Lemma \ref{lemma:limits}). We know that the fibre $X^{\bullet} \rightarrow \mathsf{pt}^{\bullet}$ is contractible: as we have remarked above, the space of factorisations through $\bigoplus_{i \in J} N_i^{\bullet} \rightarrow \bigoplus_{i \in I} N_i^{\bullet}$ is either empty or contractible. We have shown that in the case of $f$, a factorisation exists levelwise. Hence we have a fibrewise equivalence $X^{\bullet} \rightarrow \mathsf{pt}^{\bullet}$.

Taking the limit of the co-simplicial diagram above, we obtain a cartesian diagram of space
\[
\xymatrix{
\mathsf{pt} \ar[r] \ar[d] & \mathsf{pt} \ar[d]^f\\
\Map_{|\QC(\A_X^{\bullet})|}(M,\bigoplus_{i \in J} N_i^{\bullet}) \ar[r] & \Map_{|\QC(\A_X^{\bullet})|}(M,\bigoplus_{i \in I} N_i^{\bullet}),
}
\]
which implies that there exists a unique factorisation as in \eqref{eqn:factorisation}.
\end{proof}

\begin{lemma}\label{lemma:int_coproducts}
The functor $\int_X$ commutes with small colimits.
\end{lemma}
\begin{proof}
This follows from Lemma \ref{lemma:generators_coproducts}, and the fact that $-\otimes_{\Oo_X}\A_X^{\bullet}$ preserves perfect complexes, that is, compact objects (by Lemma \ref{lemma:perf_compact}).
\end{proof}

The fact that $\int$ commutes with small colimits will be used in the proof of the following lemma.

\begin{lemma}[Projection Formula]\label{lemma:projection_formula}
For every $M \in \QC(X)$, and $N^{\bullet} \in \QC(\A_X^{\bullet})$ we have an equivalence in $\QC(X)$
$$\int_X M \otimes_{\Oo_X} N^{\bullet}  \simeq M \otimes_{\Oo_X} \int_X N^{\bullet}.$$
\end{lemma}

\begin{proof}
Since $\int_X$ commutes with small colimits by Lemma \ref{lemma:int_coproducts}, we may assume without loss of generality that $M \in \Perf(X)$. In Lemma \ref{lemma:intHom} we have seen that under the equivalence $\QC(X)  \simeq \Ind \Perf(X)$, the functor $\int_X$ equivalent to $\Hhom_{\A_X^{\bullet}}(-\otimes_{\Oo_X}\A_X^{\bullet},-)$, that is, 
$\int M \otimes N^{\bullet}$ is equivalent to the functor $(-\otimes_{\Oo_X}\A_X^{\bullet})^{\vee} \otimes_{\A_X^{\bullet}} M \otimes_{\Oo_X} N^{\bullet}  \simeq M \otimes_{\Oo_X} (-\otimes_{\Oo_X}\A_X^{\bullet})^{\vee} \otimes_{\A_X^{\bullet}}N^{\bullet}$. And the right hand side is equivalent to $M \otimes_{\Oo_X} \int N^{\bullet}$.
\end{proof}

\begin{lemma}\label{lemma:unit}
The unit $\id_{\Perf(\Oo_X)} \rightarrow \int (- \otimes \A_X^{\bullet})$ is an equivalence, that is, for every $M \in \QC(\Oo_X)$ we have $M \xrightarrow{\simeq} \int(M \otimes \A_X^{\bullet})$.
\end{lemma}

\begin{proof}
This is a special case of Lemma \ref{lemma:projection_formula}, once we will have shown that $\int \A_X^{\bullet}  \simeq \Oo_X$. By definition, we have $\int_X \A_X^{\bullet}  \simeq \Hhom_{\A_X^{\bullet}}(\A_X^{\bullet},\A_X^{\bullet})  \simeq |\A_X^{\bullet}|$, where the last equivalence uses the description of Hom-spaces in limits of $\infty$-categories of Lemma \ref{lemma:limits}. By Theorem \ref{thm:beilinson} the sheaf $|\A_X^{\bullet}|$ is canonically equivalent to $\Oo_X$, by means of the augmentation $\Oo_X \rightarrow \A_X^{\bullet}$.
\end{proof}

Consequently we see that $-\otimes_{\Oo_X} \A_X^{\bullet}$ is fully faithful.

\begin{corollary}
The functor $-\otimes_{\Oo_X} \A_X^{\bullet}\colon \QC(X)  \simeq \Ind \Perf(X) \hookrightarrow \QC(\A_X^{\bullet})$ is fully faithful.
\end{corollary}

We will also use a minor variations of the notions introduced in the last two paragraphs. Let $|Z|$ be a closed subset of $X$, and denote by $\A_{X,|Z|}^{\bullet}$ the co-simplicial sheaf of algebras, given by setting $T_{\bullet} = |Z|_{\bullet} \subset |X|_{\bullet}$ in Proposition \ref{prop:co-simplicial}. There is a natural morphism of sheaves of algebras $\A_X^{\bullet} \rightarrow \A_{X,|Z|}^{\bullet}$. We will also need to consider the full subcategory $|\Perf_{|Z|}(\A_X^{\bullet})|$, consisting of objects, whose restriction (in the sense of sheaves) to $|\Perf(\A_{X \setminus Z}^{\bullet})|$ is $0$.

\begin{lemma}\label{lemma:relative}
The functor $- \otimes_{\A_X^{\bullet}} \A_{X,|Z|}^{\bullet}\colon |\Perf_{|Z|}(\A_X^{\bullet})| \hookrightarrow |\Perf(\A_{X,|Z|}^{\bullet})|$ is fully faithful.
\end{lemma}

\begin{proof}
The morphism of co-simplicial rings $\A_X^{\bullet} \rightarrow \A_{X,|Z|}^{\bullet}$ has the property that it is levelwise the projection to a factor ring. That is, for each $n \geq 0$, we have a decomposition $\A_X^n \simeq \A_{X,|Z|_n}^n \times \A_{X,|X|_n\setminus |Z|_n}^n$. By induction one shows that for $M^{\bullet} \in |\Perf_{|Z|}(\A_X^{\bullet})|$, $M^n \otimes_{\A_X^n} \A_{X,|X|_n\setminus |Z|_n}^n \simeq 0$.

This holds for $n = 0$ by virtue of the assumption that $M^0|_{X \setminus Z} \simeq 0$. We assume by induction that it holds for $k \leq n$, and we will show that $M^{n+1} \otimes_{\A_X^{n+1}} \A_{X,|X|_{n+1}\setminus |Z|_{n+1}}^{n+1} \simeq 0$. To see this, we observe that $\A_{X,|X|_{n+1}\setminus |Z|_{n+1}}^{n+1}$ decomposes as a product $\A_{X,T_1}^{n+1} \times \A_{X,T_2}^{n+1}$. The subset $T_1 \subset |X|_{n+1}$ consists of all chains $(x_0 \leq \cdots \leq x_{n+1})$, with $(x_0 \leq \cdots \leq x_n) \in |X|_n \setminus |Z|_n$, and $x_{n+1} \in |Z|$; and $T_2 \subset |X|_{n+1}$ consists of chains $(x_0 \leq \cdots \leq x_{n+1})$ with $x_{n+1} \in |X| \setminus |Z|$. We have equivalences 
$$M^{n+1} \otimes_{\A_X^{n+1}} \A_{X,T_1}^{n+1} \simeq M^n \otimes_{\A_X^{n}} \A_{X,|X|_{n}\setminus |Z|_{n}}^{n} \otimes_{\A_{X,|X|_{n}\setminus |Z|_{n}}^{n}} \A_{X,T_1}^{n+1} \simeq 0,$$
and
$$M^{n+1} \otimes_{\A_X^{n+1}} \A_{X,T_2}^{n+1} \simeq M^0 \otimes_{\A_X^{0}} \A_{X,|X|_{0}\setminus |Z|_{0}}^{0} \otimes_{\A_{X,|X|_{0}\setminus |Z|_{0}}^{0}} \A_{X,T_2}^{n+1} \simeq 0$$
by the induction hypothesis. This implies the vanishing assertion.

We have seen that for every $n \geq 0$ the functor $-\otimes_{\A_X^n} \A_{X,|Z|_n}^n$ is fully faithful, when restricted to the full subcategory of $\Perf(\A_X^n)$, consisting of objects which are equivalent to $0$, after tensoring with $- \otimes_{\A_X^n} \A_{X,|X|_n \setminus |Z|_n}^n$. That is, for two objects $M^{\bullet}\text{, }N^{\bullet}$ in this category, we have an equivalence
$$\Map(M^{n},N^{n}) \simeq \Map(M^{n}\otimes_{\A_X^n} \A_{X,|Z|_n}^n,N^{n}\otimes_{\A_X^n} \A_{X,|Z|_n}^n)$$
for every $n \geq 0$. This implies that $\lim_{[n] \in \Delta} \Map(M^{n},N^{n}) \rightarrow \lim_{[n] \in \Delta}  \Map(M^{n}\otimes_{\A_X^n} \A_{X,|Z|_n}^n,N^{n}\otimes_{\A_X^n} \A_{X,|Z|_n}^n)$ is an equivalence too. Using Lemma \ref{lemma:limits}, we identify both sides with the mapping spaces in the limit $\infty$-categories and see that the functor 
 $- \otimes_{\A_X^{\bullet}} \A_{X,|Z|}^{\bullet}\colon |\Perf_{|Z|}(\A_X^{\bullet})| \hookrightarrow |\Perf(\A_{X,|Z|}^{\bullet})|$ is fully faithful.
\end{proof}

\begin{corollary}\label{cor:relative}
Let $\int_{X,|Z|}\colon \QC( \A_{X,|Z|}^{\bullet}) \rightarrow \QC(X)$ be the right adjoint to $- \otimes_{\Oo_X} \A_{X,|Z|}^{\bullet}$. We have an equivalence of functors
$$\int_{X,|Z|} \circ (- \otimes_{\A_X^{\bullet}}\A_{X,|Z|}^{\bullet}) \simeq \int_X \colon |\Perf_{|Z|}(\A_X^{\bullet})| \rightarrow \QC(X).$$
\end{corollary}

\begin{proof}
For $M^{\bullet}$ in $|\Perf_{|Z|}(\A_X^{\bullet})|$, we have $\int_X M^{\bullet} \in \QC_Z(X)$. Therefore, it suffices to consider only perfect complexes in $X$ with set-theoretic support $|Z|$. The right hand side functor agrees with $\Map(-\otimes_{\Oo_X}\A_X^{\bullet},-)\colon\Perf_Z(X)^{\op} \times |\Perf_{|Z|}(\A_{X}^{\bullet})| \rightarrow \Grpd$. By Lemma \ref{lemma:relative} it is equivalent to $\Map(-\otimes_{\Oo_X}\A_{X,|Z|}^{\bullet},-\otimes_{\A_{X}^{\bullet}} \A_{X,|Z|}^{\bullet})$, hence to $\int_{X,|Z|} \circ (- \otimes_{\A_X^{\bullet}}\A_{X,|Z|}^{\bullet})$.
\end{proof}

\subsubsection{Modules over product rings}\label{infinite_products}

This paragraph relies on Bhatt's treatment of perfect complexes over product rings $R = \prod_{i \in I} R_i$, where $I$ is a small set \cite[Section 7]{Bhatt:2014aa}. At first we review his characterisation of perfect $R$-complexes.

\begin{theorem}[Bhatt]\label{thm:bhatt_products}
Let $I$ be a small set, and $\{R_i\}_{i \in I}$ a small family of commutative rings. We denote the ring $\prod_{i \in I}$ by $R$. Let $\{M_i\}_{i \in I}$ be a family of perfect $R_i$-complexes. We say that $\{M_i\}_{i \in I}$ is globally bounded, if there exist two integers $m\text{, }n \in \mathbb{Z}$, such that $H^j(M_i) \neq 0$ implies $m \leq j \leq n$, for all $i \in I$.
Then the functor $\Perf(R) \rightarrow \prod_{i \in I}\Perf(R_i)$ induced by tensoring along the projection maps $p_i \colon R \rightarrow R_i$, induces an equivalence with the full subcategory of globally bounded families of perfect complexes.
\end{theorem}

In particular this theorem implies that we have a fully faithful functor $\Perf(R) \hookrightarrow \prod_{i \in I} \Perf(R_i)$, and moreover it characterises the essential image of this embedding. The functor $\QC(R) \rightarrow \prod_{i \in I} \QC(R_i)$ is not fully faithful. In the following example we illustrate this behaviour by alluding to classical, non-derived tensor products and modules over infinite product rings.

\begin{example}
Let $I$ be an infinite small set; and $\{R_i\}_{i \in I}$ a collection of non-trivial rings. Let $R = \prod_{i \in I} R_i$ be the infinite product. We define an equivalence relation $\sim$ on $R$, where we say that $(x_i)_{i \in I} \sim (y_i)_{i \in I}$, if and only if $x_i = y_i$ for all but finitely many indices $i$. We define the $R$-module $M$ to be $R/\sim$. By construction, we have $M \otimes_R R_i = 0$ for all $i \in I$, although $M \neq 0$.
\end{example}

This failure of fidelity is the reason that we cannot expect the functor $\QC(X) \rightarrow |\QC(\Ab_X^{\bullet})|$ to be an equivalence. This remark will become clear by inspecting the argument of Lemma \ref{lemma:minimal} below.

\subsubsection{Conservativity}

The following assertion is key in our argument that $-\otimes \A_X^{\bullet}$ is an equivalence of categories, with inverse $\int$.
\begin{proposition}\label{prop:conservative}
If $X$ is affine, then the functor $\int\colon\Perf(\A_X^{\bullet}) \rightarrow \QC(X)$ is conservative. 
\end{proposition}

The proof is scattered over this paragraph, and broken down into several lemmas. Our argument relies on two technical results: an invariance property for derived $\infty$-categories with set-theoretic support condition, given by Thomason--Trobaugh Theorem \cite[Theorem 2.6.3]{MR1106918}; and the theory of modules over product rings, described in Paragraph \ref{infinite_products}. For the convenience of the reader, we have allowed ourselves to state Thomason--Trobaugh's result in a way which can be directly applied to our context. We refer the reader to the original reference for a more general statement.

\begin{theorem}[Thomason--Trobaugh]\label{thm:thomason-trobaugh}
Let $R$ be a Noetherian ring, and $I$ an ideal, and $\Perf_I(R)$ denote the full subcategory of $\Perf(R)$ consisting of perfect complexes set-theoretically supported at the closed subset of $\Spec R$ corresponding to $I$. We denote by $\widehat{R}_I$ the completion of $R$ at $I$. Then, tensoring along $R \rightarrow \widehat{R}_I$ induces an equivalence of symmetric monoidal $\infty$-categories
$\Perf_I(R) \cong \Perf_{\widehat{I}}(\widehat{R}_I).$
\end{theorem}

From now on we fix $0 \neq N^{\bullet} \in \Perf(\A_X^{\bullet})$; and will show that $\int N^{\bullet} \neq 0$. 

\begin{lemma}\label{lemma:minimal}
If $X$ is affine, and $N^{\bullet} \in \Perf(\A_X^{\bullet})$, then there exists a closed point $x \in X$ such that $N^0 \otimes_{\A_X^0} \widehat{\Oo_x} \neq 0$.
\end{lemma}
\begin{proof}
It follows from Bhatt's theorem \ref{thm:bhatt_products} that there is a (not necessarily closed) point $y \in X$ with $N^0 \otimes_{\A_X^0} \widehat{\Oo_y} \neq 0$.  Without loss of generality we may assume that $N^0 \otimes_{\A_X^0} \widehat{\Oo_y} \in \Perf_{\{y\}}(\widehat{\Oo}_y)$, otherwise choose $F \in \Perf_{\{y\}}(\widehat{\Oo}_y)$, such that $N^0 \otimes_{\A_X^0} \widehat{\Oo_y} \otimes_{\widehat{\Oo}_y} F \neq 0$ (existence is guaranteed by Lemma \ref{lemma:F}), and consider $N^{\bullet} \otimes_{\Oo_X}F$ instead. The latter is fine, because $N^0 \otimes_{\A_X^0} \widehat{\Oo_x} \otimes_{\Oo_x} F \neq 0$ implies $N^0 \otimes_{\A_X^0} \widehat{\Oo_x} \neq 0$. Note that $N^{\bullet} \otimes_{\Oo_X}F \in \QC(\A_X^{\bullet})$, because $F$ might not be perfect as a complex of $\Oo_X$-module. 

Let $x \in \overline{\{y\}}$ be a closed point. 
Observe that such an $x$ exists, because $X$ is affine (every ideal is contained in a maximal ideal). We claim that $N^0 \otimes_{\A_X^0} \widehat{\Oo_x} \neq 0$ as well. Indeed, we have a map co-simplicial sheaves of algebras $f\colon \A_X^{\bullet} \rightarrow \A_{X,T_{\bullet}}^{\bullet}$, where $T_{\bullet}\subset |X|_{\bullet}$ is the minimal simplicial subset containing the chains $(x)$, $(y)$, and $(x < y)$. 

By \cite[Proposition 3.3.3]{MR1138291}, the ring $\A^{1}_{X,T^1}$ is isomorphic to a product $\A^0 \times \A(x < z)$, where the first factor corresponds to the degenerate part of the co-simplicial ring.

The maps $\A_X^{0} \rightarrow \widehat{\Oo}_z$ for $z = x,y$ factor through $\A^0_{X,T^0}(\Oo_X)$.
We have the commutative diagram (see \cite[Proposition 3.2.1]{MR1138291} for how $\A(x < y)$ fits in there)
\[
\xymatrix{
\Oo_y \ar[r] \ar[d] & \widehat{\Oo}_x \otimes_{\Oo_X} \Oo_y \ar[d] & \widehat{\Oo}_x \ar@{=}[d] \ar[l] \\
\widehat{\Oo}_y \ar[r] & \A(x < y) & \widehat{\Oo}_x. \ar[l]
}
\]
We assume by contradiction that $N^{\bullet} \in \QC_{\{y\}}(\Oo_y) \cong \QC_{\{y\}}(\widehat{\Oo}_y)$ (where the equivalence uses Thomason--Trobaugh's Theorem \ref{thm:thomason-trobaugh}), and $N^{0} \otimes_{\A^0} \widehat{\Oo}_x = 0$. The commutative diagram above implies
$$(N^{0} \otimes_{\A^0} \widehat{\Oo}_y) \otimes_{\Oo_y} (\widehat{\Oo}_x \otimes_{\Oo_X} \Oo_y) = 0.$$ The ring homormophism $\Oo_y \rightarrow \widehat{\Oo}_x \otimes_{\Oo_X} \Oo_y$ is the co-base change of the fully faithful ring homomorphism $\Oo_x \rightarrow \widehat{\Oo}_x$. Therefore, we have $N^0 \otimes_{\A^0} \widehat{\Oo}_y = 0$, which is a contradiction.
\end{proof}

The following lemma could also be deduced from a more elementary argument. We give a more abstract proof, which uses a categorical viewpoint on ring completions.

\begin{lemma}\label{lemma:F}
Let $X$ be an affine Noetherian scheme. For every $0 \neq N^{\bullet} \in \Perf(\A_X^{\bullet})$ there exists a closed point $x \in X$, and a perfect complex $F \in \Perf_{\{x\}}(X)$, such that $F \otimes_{\Oo_X} N^{\bullet} \neq 0$.
\end{lemma}
\begin{proof}
We have seen in Lemma \ref{lemma:minimal} that there exists a closed point $x \in X$, such that the perfect complex of $\widehat{\Oo_x}$-modules $N^{0} \otimes_{\A_X^0} \widehat{\Oo}_x$ is non-zero. The Yoneda embedding restricts to a fully faithful functor $\Perf(\widehat{\Oo}_x) \hookrightarrow \Ind \Perf_{\{x\}}(X)$ (see the reformuation of \cite{Efimov:2010fk} in \cite[Proposition 5.2 \& 5.4]{Braunling:2014aa}). Therefore, there exists an object $F \in \Perf_{\{x\}}(X)$, such that $\Hhom(F^{\vee},N^{\bullet} \otimes_{\A_X^0} \widehat{\Oo}_x)  = F \otimes_{\Oo_X} N^{0} \otimes_{\A_X^0} \widehat{\Oo}_x\neq 0$. This implies $F \otimes_{\Oo_X} N^{\bullet} \neq 0$.
\end{proof}

In the following we denote by $\Hc^i(-)$ the degree $i$ cohomology sheaf of a complex of sheaves. Recall that we have functors $\A_{X,T}(-)\colon \QCoh(X) \rightarrow \Mod(\Oo_X)$, as defined in Definition \ref{defi:adeles}. As before, we denote by $\A_X^0(-)$ the functor for $T = |X|$.
 
\begin{lemma}\label{lemma:cohomology_sheaves}
Let $M \in \QC(X)  \simeq \Ind \Perf(X)$, then $\Hc^i(M \otimes_{\Oo_X} \A_{X,T})  \simeq \A_{X,T}(\Hc^i(M))$.
\end{lemma}
\begin{proof}
By definition, the functor $\A_{X,T}(-) \simeq \QCoh(X) \rightarrow \Mod(\Oo_X)$ commutes with filtered colimits. Since the same is true for $\Hc^i(-)$, and $- \otimes_{\Oo_X}\A_{X,T}(\Oo_X)$, we see that we may assume that $M$ is a perfect complex. Since the statement is local, we may assume that $M$ can be represented by a bounded complex $(V^{\bullet})$ of free sheaves of finite generation $V^i = \Oo_X^{m_i}$. In particular, $\Hc^i$ is represented as the cohomology sheaf in degree $i$ of the chain complex
$$\cdots \rightarrow \Oo_X^{m_{i-1}} \rightarrow \Oo_X^{m_{i}} \rightarrow \Oo_X^{m_{i+1}} \rightarrow \cdots.$$
Since $\A_{X,T}(-)$ is exact, we obtain a chain complex 
$$\cdots \rightarrow \A_{X,T}(V\Oo_X^{m_{i-1}}) \rightarrow \A_{X,T}(\Oo_X^{m_{i}}) \rightarrow \A_{X,T}(\Oo_X^{m_{i-1}}) \rightarrow \cdots,$$
and an equivalence of middle cohomology sheaves. The second chain complex can be identified with $M \otimes_{\Oo_X} \A_{X,T}$. This implies the assertion.
\end{proof}

\begin{proof}[Proof of Proposition \ref{prop:conservative}]
Let $0 \neq N^{\bullet} \in \Perf(\A_X^{\bullet})$, and $F \in \Perf(X)$ as in Lemma \ref{lemma:F}. In particular, there exists a closed point $x$, such that $F \in \Perf_{\{x\}}(X)$. By the projection formula (Lemma \ref{lemma:projection_formula}) we have 
$$\int(N^{\bullet} \otimes_{\Oo_X} F) \cong F \otimes_{\Oo_X} \int N^{\bullet}.$$
Therefore, it suffices to show that this expression is $\neq 0$, to deduce that $\int N^{\bullet} \neq 0$. Henceforth, we replace $N^{\bullet}$ by $N^{\bullet} \otimes F$, that is, we have $N^{\bullet}|_{X \setminus \{x\}} = 0$. Since pullback along $\Spec \widehat{\Oo}_x \rightarrow X$ induces an equivalence of $\infty$-categories (using Thomason--Trobaugh's Theorem \ref{thm:thomason-trobaugh})
$$\Perf_{\{x\}}(X) \rightarrow \Perf_{\{x\}}(\widehat{\Oo}_x),$$
one obtains $N^0 \cong M \otimes_{\Oo_X} \A_X^0$, for some $M \in \Perf_{\{x\}}(X)$. Let $d_i\colon \A_X^0 \rightarrow \A_X^n$ be an arbitrary boundary map, and $\phi_i \colon \Oo_X \rightarrow \A_X^n$ the map given by composition with the augmentation. We claim that $M \rightarrow M \otimes \A_X^n$ is a quasi-isomorphism of complexes of $\Oo_X$-modules. Indeed, passing to the cohomology sheaves, we obtain using Lemma \ref{lemma:cohomology_sheaves}
$$\Hc^i(M) \rightarrow \A_X^{n}(\Hc^i(M)).$$
This is an isomorphism, since $\Hc^i(M)$ is set-theoretically supported at the closed point $x$ (Remark \ref{rmk:zero-dimensional}). 

Hence we see that $N^{\bullet} \otimes_{\Oo_X} F$ is equivalent to the constant object $M \otimes_{\Oo_X}\A_X^{\bullet}$ in $\Perf(\A_X^{\bullet})$, which implies $\int N^{\bullet} \otimes_{\Oo_X} F \cong N^0 \otimes_{\Oo_X} F \neq 0$.
\end{proof}

\subsubsection{Conclusion of the argument}\label{sub:conclusion}

Henceforth we assume that $X$ is an affine Noetherian scheme. In this paragraph we will prove that for $M^{\bullet} \in |\Perf(\A_X^{\bullet})|$ the object $\int M^{\bullet}$ is also perfect. We will first show (Corollary \ref{cor:generic}) that this statement holds generically, that is, there exists an open subset $U \subset X$, such that $\int M^{\bullet}|_{U} \in \Perf(U)$. We use this result as a stepping stone, together with the Noetherian property of $X$, to show that $\int M^{\bullet}$ is in fact a perfect complex. The key observation to deduce this, is Lemma \ref{lemma:generic_zero}, which asserts that for a generic point $\eta$ of $X$, $M^0 \otimes_{\A_X^0} \Oo_{\eta} \simeq 0$ implies the existence of an open subset $V \subset X$, such that $M^{\bullet}|_V \simeq 0$.

\begin{lemma}\label{lemma:homotopy_equivalence}
Let $X$ be a Noetherian scheme, and $|Z| \subset X$ a closed subset, with $\eta$ a generic point of $|Z|$, and $U$ an affine open neighbourhood of $\eta$, such that $U \cap Z$ is irreducible. We have that the canonical augmentation $\Oo_{\eta}|_U \rightarrow \A_{X,|Z|_{\bullet}}^{\bullet}(\Oo_{\eta})|_U$ is a homotopy equivalence of co-simplicial sheaves of algebras.
\end{lemma}

We refer the reader to \cite[Tag 019U]{stacks-project} for a detailed review of the definition of a homotopy between maps of co-simplicial objects in a category $\C$. In brief, we say that $f\text{, }g\colon U^{\bullet} \rightarrow V^{\bullet}$ are homotopic, if for every $n \geq 0$, and every $\alpha\colon [n] \rightarrow [1]$ we have maps $h^{n,\alpha}\colon U^n \rightarrow V^n$, such that for $f\colon [k] \rightarrow [n]$ the diagram
\[
\xymatrix{
U^k \ar[r]^f \ar[d]_{h^{k,\alpha \circ f}} & U^n \ar[d]^{h^{n,\alpha}} \\
V^k \ar[r]^f & V^n
}
\]
commutes. We also require that the original maps $f^n$ and $g^n$ are obtained as $h^{n,0}$ and $h^{n, 1}$, where $i\colon [n] \rightarrow [1]$ is the constant map with value $i$, for $i \in [1]$. 

There is an equivalent definition, which makes sense in any $\infty$-category, possessing finite products, so in particular in the $\infty$-category of small $\infty$-categories. We refer the reader to Definition \ref{defi:infty_homotopy}. In Lemma \ref{lemma:cat_homotopy} we show that two homotopy equivalent co-simplicial objects in $\infty$-categories, have canonically equivalent limit $\infty$-categories. 

There is also a notion of homotopy for maps between simplicial sets. Let $S_{\leq}$ be a partially ordered set, with a maximal element $\eta$, and we will mainly care about $S_{\leq} = |U\cap Z|_{\leq}$ with the generic point $\eta$. We have the simplicial set of ordered chains $S_{\bullet}$, where $S_n = \{x_0\leq \cdots \leq x_n\}$. We have a homotopy
$$h_{n,\alpha}(x_0 \leq \cdots \leq x_n) = x_0\leq \cdots \leq x_k \leq \eta \cdots \leq \eta,$$
where $k$ is the maximal element of $[n]$, such that $\alpha(k) = 0$. This defines a homotopy between the identiy of the simplicial set $S_{\bullet}$ of ordered chains, and the constant map $S_{\bullet} \rightarrow \{\eta\} \subset S_{\bullet}$. In other words, the constant simplicial set $\{\eta\}_{\bullet}$ is a deformation retract of $S_{\bullet}$. This is a minor variation of \cite[Tag 08Q3]{stacks-project}. The proof of the Lemma above will use this homotopy $h_{n,\alpha}$, which contracts $|U\cap Z|_{\bullet}$ onto the generic point.

\begin{proof}[Proof of Lemma \ref{lemma:homotopy_equivalence}]
From now on we replace $X$ by $U$ without loss of generality. Recall that the co-simplicial algebra $\A_{X,|Z|}^{\bullet}(\Oo_{\eta})$ is a subobject of the co-simplicial algebra $[n] \mapsto \prod_{[n] \rightarrow |Z|_{\leq}} \A_{X,\sigma}(\Oo_{\eta})$ (Lemma \ref{lemma:local_factors}). For $\sigma\colon [n] \rightarrow |Z|_{\leq}$ we denote by $\sigma_+$ the map $[n+1] \rightarrow |Z|_{\leq}$, which restricts to $\sigma$ along the map $[n] \rightarrow [n+1]$, and sends $n+1$ to $\eta$. By definition, we have $\A_{X,\sigma}(\Oo_{\eta}) = \A_{X,\sigma_+}(\Oo_X)$. Therefore, we have for $\alpha\colon [n] \rightarrow [1]$ that the chain of points corresponding to $h_{n,\alpha}(\sigma)_+$ can be extended to the chain $\sigma_+$, that is, there is a canonical inclusion
$$i_{\sigma,\alpha}\colon \A_{X,h_{n,\alpha}(\sigma)}(\Oo_{\eta}) \rightarrow \A_{X,\sigma}(\Oo_{\eta}).$$
We define the homotopy as a map
$$h^{n,\alpha}\colon \prod_{\sigma\colon [n] \rightarrow |Z|_{\leq}}\A_{X,\sigma}(\Oo_{\eta}) \rightarrow \prod_{\sigma\colon [n] \rightarrow |Z|_{\leq}}\A_{X,\sigma}(\Oo_{\eta}),$$
sending $(x_{\sigma})_{\sigma\colon [n] \rightarrow |Z|_{\leq}}$ to $(i_{\sigma,\alpha}(x_{h_{n,\alpha}(\sigma)}))_{\sigma\colon [n] \rightarrow |Z|_{\leq}}$. It remains to verify that the homotopy $h^{n,\alpha}$ respects the subobject $\A_{X,|Z|}^{\bullet}(\Oo_{\eta})$.

To see this for a given $\alpha\colon [n] \rightarrow [1]$, we observe that the ring homomorphism $h^{n,\alpha}\colon \A_{X}^n(\Oo_{\eta}) \rightarrow \A_{X}^n(\Oo_{\eta})$ can be obtained using the co-simplicial structure $\A_X^{\bullet}(\Oo_{\eta})$. For $\alpha\colon [n] \rightarrow [1]$, we denote by $[k] = \alpha^{-1}(0)_+$ the preimage of $0$, with an additional maximal element $\eta$ adjoined (otherwise $k$ could be empty). We let $s\colon [k] \rightarrow [n]_+$ be the inclusion, which identifies $[k] \setminus\{\eta\}$ with $\alpha^{-1}(0)$, and preserves maxima. We define $r\colon [n]_+ \rightarrow [k]$ to be the unique map in $\Delta$, preserving maxima, and satisfying $r\circ s = \id_{[k]}$. We let $p= s \circ r\colon [n]_+ \rightarrow [n]_+$, and define $T = \{(x_0\leq \cdots \leq x_{n+1} \in |X|_{n+1}| x_{n+1} = \eta)\}$. By virtue of \cite[Section 2.2 \& 2.3]{MR1138291} we have a morphism $\A_X^n(\Oo_{\eta}) = \A_{X,T}^{n+1}(\Oo_X) \rightarrow \A_{X,T}^{n+1}(\Oo_X) = \A_X^{n}(\Oo_{\eta})$, which agrees with $h^{n,\alpha}$ on the local factors. This shows that the homotopies $h^{n,\alpha}$ are well-defined.

Hence, we have constructed a homotopy between $\id_{\A_{X,|Z|}^{\bullet}(\Oo_{\eta})}$ and the composition $\A_{X,|Z|}^{\bullet}(\Oo_{\eta}) \rightarrow (\Oo_{\eta})^{\bullet} \rightarrow \A_{X,|Z|}^{\bullet}(\Oo_{\eta})$.
\end{proof}

Since homotopy equivalent co-simplicial rings have equivalent categories of cartesian modules by Lemma \ref{lemma:cat_homotopy}, we obtain the following consequence.

\begin{corollary}\label{cor:pre_generic}
For $X$ an affine Noetherian scheme, with a closed subset $|Z| \subset X$, and $M^{\bullet} \in |\Perf_{|Z|}(\A_{X}^{\bullet})|$ (that is, $M^{\bullet}|_{X\setminus |Z|} \simeq 0$, see also Lemma \ref{lemma:relative}), for every generic point $\eta$ of $|Z|$, the stalk $(\int M^{\bullet})_{\eta}$ is a perfect $\Oo_{\eta}$-complex.
\end{corollary}

\begin{proof}
By the Projection Formula \ref{lemma:projection_formula} we have $(\int M^{\bullet})_{\eta} \simeq \int (M^{\bullet} \otimes_{\Oo_X} \Oo_{\eta})$. The co-simplicial object $M^{\bullet} \otimes_{\Oo_X} \Oo_{\eta}$ can be viewed as an $\A_{X,|Z|_{\bullet}}^{\bullet}(\Oo_{\eta})$-module by Lemma \ref{lemma:relative}. Since the co-simplicial sheaf of algebras $\A_{X,|Z|_{\bullet}}^{\bullet}(\Oo_{\eta})$ is homotopy equivalent to the constant co-simplicial sheaf of algebras $\Oo_{\eta}$, with homotopy inverse given by the projection $\A_{X,|Z|_{\bullet}}^{\bullet}(\Oo_{\eta}) \rightarrow \Oo_{\eta}$, we have an equivalence
$$|\Perf(\A_{X,|Z|_{\bullet}}^{\bullet}(\Oo_{\eta}))| \rightarrow |\Perf(\Oo_{\eta})|,$$
inverse to the the realisation functor given by tensoring along the augmentation $\Oo_{\eta}\rightarrow \A_{X,|Z|_{\bullet}}^{\bullet}(\Oo_{\eta})$. By the adjunction of $\int$ and $- \otimes_{\Oo_X} \A_X^{\bullet}$, this inverse agrees with $\int$ (see also Corollary \ref{cor:relative}). This implies that $(\int M^{\bullet})_{\eta}$ is a perfect $\Oo_{\eta}$-complex.
\end{proof}

\begin{lemma}\label{lemma:generic_zero}
Let $X$ be an affine Noetherian scheme, $|Z| \subset X$ a closed subset, and $M^{\bullet} \in |\Perf(\A_{X,|Z|_{\bullet}}^{\bullet})|$, such that for $\eta$ a generic point of $|Z|$, we have $M^0 \otimes_{\A_{X,|Z|}^0} \Oo_{\eta} \simeq 0$. Then, there exists an open subset $V \subset X$, such that $M^{\bullet}|_V \simeq 0$.
\end{lemma}

\begin{proof}
Let $\A_{X,|Z|,< \eta}^0 = \prod_{x < \eta} \widehat{\Oo}_{x}$, which is a factor of $\A_{X,|Z|}^0 = \prod_{x \in |Z|} \widehat{\Oo}_x$. By assumption, $\A_{X,|Z|,< \eta}^0 \otimes_{\Oo_X} \Oo_{\eta}$ is a factor of $\A_{X,|Z|}^1$. Via the co-simplicial structure, the second ring receives a map from $\A_{X,|Z|,< \eta}^0$, and also from $\Oo_{\eta}$ (another factor of $\A_{X,|Z|}^0$):
\[
\xymatrix{
\A_{X,|Z|,< \eta}^0 \ar[rd] & &\\
& \A_{X,|Z|,< \eta}^0 \otimes_{\Oo_X} \Oo_{\eta} \ar@{^(->}[r] & \A^1_{X,|Z|}. \\
\Oo_{\eta} \ar[ru] & &
}
\]
The assumption $M^0 \otimes_{\A_{X,|Z|}^0} \Oo_{\eta} \simeq 0$ implies therefore also $M^1 \otimes_{\A_{X,|Z|}^1} (\A_{X,|Z|,< \eta}^0 \otimes_{\Oo_X} \Oo_{\eta}) \simeq 0$. This yields $(M^0 \otimes_{\A_{X,|Z|}^0} \A_{X,|Z|,< \eta}^0) \otimes_{\Oo_X} \Oo_{\eta} \simeq 0$. 
We can write $\Oo_{\eta}$ as a filtered colimit of rings $\colim_{f \in \Oo_X} \Oo_{X,f}$, and hence see that $\A_{X,|Z|,< \eta}^0 \otimes_{\Oo_X} \Oo_{\eta} \simeq \colim_{f \in \Oo_X} \A_{X,|Z|,< \eta}^0 \otimes_{\Oo_X} \Oo_{X,f}$ is a filtered colimit of rings. 

By \cite[Proposition 3.20]{MR1106918}, if $R = \colim_{i \in I} R_i$ is a filtered colimit of rings, and for $j \in I$, $M_j \in \Perf(R_j)$, such that $M_j \otimes_{R_j} R \simeq 0$, then there exists $i \geq j$, such that $M_j \otimes_{R_j} R_i \simeq 0$ (in \emph{loc. cit.} it is shown that the homotopy category of perfect complexes of $R$-modules is equivalent to the filtered colimit of the homotopy category of perfect complexes of $R_i$-modules).

Since our sheaves of algebras are l\^ache, we may apply this result, since by virtue of Corollary \ref{cor:eq_lache} the categories over perfect complexes over these sheaves of algebras are equivalent to perfect complexes over the rings of global sections. Hence, we see that there exists $f \in \Oo_{X}$, such that $(M^0 \otimes_{\A_{X,|Z|}^0} \A_{X,|Z|,< \eta}^0) \otimes \Oo_f \simeq 0$. Restricting both sides to the open subset $V \subset X$ defined by localisation at $f$, we obtain $M^0|_V \simeq 0$, because $/(\Oo_f)|_V \cong \Oo$.
\end{proof}

The following non-example demonstrates how the property of being $0$ at the generic point \emph{spreads out} to an open neighbourhood.

\begin{example}
Consider the example $X=\Spec \mathbb{Z}$, we then have a perfect $\Ab_X^0$-complex given by $\prod_{p \text{ prime}} \mathbb{F}_p$ by Bhatt's Theorem \ref{thm:bhatt_products} for perfect complexes over product rings. We claim that this perfect complex cannot be $M^0$ of an adelic descent datum $M^{\bullet} \in |\Perf(\A_X^\bullet)|$. Indeed, if it was, then we would have $M^0 \otimes_{\Oo_X} \mathbb{Q} \simeq 0$. But the element $(1,\cdots)_{p\text{ prime}} \in \Hc^0(M)$, is not annihilated by any $f \in \Oo_X(X) = \mathbb{Z}$.
\end{example}

We will now put this spreading-out property of adelic descent data to use, to show that for $M \in |\Perf(\A_X^{\bullet})|$ the complex $\int M$ is generically perfect.

\begin{corollary}\label{cor:generic}
Let $X$ be an affine Noetherian scheme, $|Z| \subset X$ a closed subset, and $M^{\bullet} \in |\Perf(\A_{X,|Z|}^{\bullet})|$. For every generic point $\eta$ of $|Z|$ exists an open neighbourhood $U \subset X$, such that $(\int M^{\bullet})|_U$ is a perfect complex on $U$.
\end{corollary}

\begin{proof}
Choose an isomorphism of $(\int M^{\bullet}) \rightarrow \colim_{i \in I} N_i$ with a filtered colimit of perfect complexes on $X$. By Corollary \ref{cor:pre_generic}, we know that $(\int M)_{\eta}$ is a perfect $\Oo_{\eta}$-complex. Therefore we see that there exists an index $j \in I$, such that we have a factorisation 
\[
\xymatrix{
(\int M^{\bullet})_{\eta} \ar[r]^s \ar[rd] & (N_j)_{\eta} \ar[d]^r, \\
& (\int M^{\bullet}) 
}
\]
and we see that $p = s \circ r$ defines an idempotent on $(N_j)_{\eta}$, such that the corresponding direct summand is equivalent to $(\int M^{\bullet})_{\eta}$. Since $N_j$ is perfect, there exists an affine open subset $U \subset X$, and an idempotent $p$ on $N_j|_U$, which extends the one over $\Oo_{\eta}$. Passing to the corresponding direct summand $N'_j$ of $N_j|_U$, we see that we have constructed a direct complex $N'_j$, with a map to $(\int M^{\bullet})|_U$, which induces an equivalence at $\eta$. By adjunction, we have a morphism $N'_j \otimes_{\Oo_X} \A_{X,|Z|}^{\bullet} \rightarrow M^{\bullet}$, which is an equivalence after tensoring with $- \otimes_{\A_{X,|Z|}^0} \Oo_{\eta}$. By Lemma \ref{lemma:generic_zero} there exists an open subset $V \subset U$, such that $N'_j \otimes_{\Oo_X} \A_{V,|Z| \cap V}^{\bullet} \rightarrow M^{\bullet}|_V$ is an equivalence. This implies that $(\int M^{\bullet})|_V$ is equivalent to $N'_j|_V$, and thus is a perfect complex.
\end{proof}

\begin{proof}[Proof of Proposition \ref{prop:adelic_descent}]
Let $M^{\bullet} \in |\Perf(\A_X^{\bullet})|$, we will show that $\int M^{\bullet}$ is perfect. Thus, by conservativity of $\int$, and the fact that $- \otimes_{\Oo_X} \A_X^{\bullet}$ is fully faithful, we can apply Lemma \ref{lemma:weakBB} to conclude the proof.

Let $\eta$ be a generic point of $X$. By virtue of Corollary \ref{cor:generic} there exists an open neighbourhood $U$, and a perfect complex $N \in \Perf(U)$, together with an isomorphism $N \otimes_{\Oo_U} \A_X^{\bullet} \simeq M|_U$. We would like to extend this map to $X$, and then take the cofibre thereof. But in general it is not possible to extend a perfect complex from an open subset, unless its $K$-theory class extends \cite[Proposition 5.2.2]{MR1106918}. Hence, we consider instead $M_0^{\bullet} = M^{\bullet} \oplus \Sigma M$, and $N'_0 \oplus \Sigma N$, whose $K$-theory class is $0$, but contain the perfect complexes we care about as a direct summand. We may assume that $N'_0 = N_0|_U$, where $N''_0 \in \Perf(X)$. Let $j\colon U \hookrightarrow X$ be the corresponding open immersion, we have a map
$$N_0'' \rightarrow j_*j^*N_0'' \rightarrow j_*j^*\int M_0^{\bullet},$$
and we see that there exists a line bundle $L$, such that we have a map $N_0 = N_0'' \otimes L \rightarrow \int M_0^{\bullet}$, which induces an equivalence, when restricted to $U$. We consider the cofibre $C$ of the adjoint $N_0 \otimes_{\Oo_X} \A_X^{\bullet} \rightarrow M_0^{\bullet}$, and define $M_1^{\bullet} = C \oplus \Sigma C$.

Since $M_1^0 \otimes_{\A_X^0} \Oo_{\eta} \simeq 0$, we see that $M_1 \in \Perf(\A_{X,|Z|}^{\bullet})$ for a closed subset $|Z| \subset X$ by Lemma \ref{lemma:generic_zero}. Iterating this process, we obtain a sequence $(M_n^{\bullet})_{n\in \mathbb{N}}$ of objects in $\Perf(\A_X^{\bullet})$, together with an increasing sequence of open subsets $U_n \subset |X|$, such that $\int M_n^{\bullet}|_{U_n}$ is perfect. Since $X$ is Noetherian, this sequence stabilises, that is $U_n = U$ for $n >> 0$. We must have $U = |X|$, because otherwise we could choose a generic point of the complement, and continue the iterative process described above. Since $M_n^{\bullet}$ contains $M^{\bullet}$ as a direct summand, we conclude that $\int M^{\bullet}$ is a perfect complex.
\end{proof}

\subsection{Adelic reconstruction}

In this subsection we collect various corollaries of our adelic descent result \ref{thm:adelic_descent}. Applying the Tannakian formalism for symmetric monoidal $\infty$-categories, we can show that a Noetherian scheme $X$ can be reconstructed from the co-simplicial ring $\Ab_X^{\bullet}$. We will also give an adelic description of $G$-torsors on $X$, for any Noetherian affine group scheme, analogously to Weil's Theorem \ref{thm:weil} for algebraic curves, which inspired the present work.

Our results rely heavily on the Tannakian formalism for symmetric monoidal stable $\infty$-categories, as developed by Lurie in \cite{lurie2011quasi}, and further refined by Bhatt \cite{Bhatt:2014aa} and Bhatt--Halpern-Leistner \cite{bhatt2015tannaka}.

\begin{theorem}[Bhatt, Bhatt--Halpern-Leistner]\label{thm:bhattdhl}
\begin{itemize}
\item[(a)] The category $\mathsf{AlgSp}^{\qcqs}$ of quasi-compact and quasi-separated algebraic spaces embeds fully faithfully into $(\St_{\otimes})^{\op}$, the dual $\infty$-category of small symmetric monoidal stable $\infty$-categories, by means of the functor $\Perf(-)_{\otimes}\colon (\mathsf{AlgSp}^{\qcqs})^{\op} \hookrightarrow \St_{\otimes}$.

\item[(b)] The functor $\APerf_{\cn}(-)_{\otimes}$ from the $\infty$-category of spectral Noetherian stacks $(\mathsf{DStck}^{\N})^{\op}$ to the $\infty$-category of symmetric monoidal $\infty$-categories $\inftyCat_{\otimes}$ is fully faithful.
\end{itemize}
\end{theorem}

Recall that $\APerf$ denotes the $\infty$-category of almost perfect, that is, pseudo-coherent complexes. That is, a complex which can locally be presented as a bounded above complex of finitely generated locally projective sheaves of modules. The subscript $\cn$ refers to the full subcategory of \emph{connective} objects, that is complexes, which are concentrated in non-positive degrees.

\subsubsection{For schemes}

By a well-known theorem of Gelfand--Naimark \cite{gelfand}, every locally compact space $X$ can be reconstructed from the ring of continuous real-valued functions. It is equally well-known that a single ring is not sufficient to capture the delicate geometry of an arbitrary non-affine scheme. The result below shows however that for a Noetherian scheme $X$, the co-simplicial ring of ad\`eles $\Ab_X^{\bullet}$ allows one to reconstruct $X$. It is intriguing to observe that local compactness for topological spaces is not dissimilar from the Noetherian hypothesis for schemes. To iterate further on this philosophical point, we remind ourselves that a Hausdorff topological vector space is finite dimensional if and only if it is locally compact. The Noetherian condition enforces finite-dimensionality of the local rings, and guarantees that ideals are finitely generated.

\begin{theorem}\label{thm:schematic_reconstruction}
For any Noetherian scheme $X$ we have a canonical equivalence $|\Spec \Ab_X^{\bullet}|  \simeq X$, where the colimit $|\cdot|$ is taken in the category of quasi-compact and quasi-separated schemes.
\end{theorem}

\begin{proof}
We use Bhatt's Tannakian reconstruction result, as presented in \cite{Bhatt:2014aa}. Theorem \ref{thm:adelic_descent} gives a symmetric monoidal equivalence $\Perf(X)_{\otimes}  \simeq |(\Perf(\Ab_X^{\bullet}))_{\otimes}|$. By \cite[Corollary 1.6]{Bhatt:2014aa} we may deduce the existence of an augmentation $\Spec \Ab_X^{\bullet} \rightarrow X$, yielding an equivalence $|\Spec \Ab_X^{\bullet}| \rightarrow X$ in the category of quasi-compact and quasi-separated schemes.
\end{proof}

Bhatt's \cite{Bhatt:2014aa} allows one to reconstruct a quasi-compact and quasi-separated algebraic space $X$ from the symmetric monoidal $\infty$-category $\Perf(X)_{\otimes}$. This result extends work by Lurie \cite{lurie2011quasi}, and the classical Tannakian philosophy in general. 

\begin{corollary}\label{cor:main4}
The functor $\Ab^{\bullet}\colon (\Sch^{\N})^{\op} \rightarrow \Rings^{\Delta}$, sending a Noetherian scheme $X$ to $\Ab_X^{\bullet}$, is faithful. 
\end{corollary}
\begin{proof}
We denote by $\St_{\otimes}$ the $\infty$-category of small symmetric monoidal $\infty$-categories. It is shown by Bhatt's Theorem \ref{thm:bhattdhl}(a) that there is a fully faithful functor $(\Sch^{\qcqs})^{\op} \rightarrow \St_{\otimes}$. As our construction has shown, the restriction of this functor to Noetherian schemes, factors through $(\Sch^{\mathsf{N}})^{\op} \rightarrow \Rings^{\Delta} \rightarrow \St_{\otimes}$ the category of co-simplicial rings. This shows that the functor $(\Sch^{\mathsf{N}})^{\op} \rightarrow \Rings^{\Delta}$ is faithful.

Let $X$ and $Y$ be Noetherian schemes, and $f^{\sharp}\colon \Ab_Y^{\bullet} \rightarrow \Ab_X^{\bullet}$ a morphism of co-simplicial rings. Since $X  \simeq |\Spec \Ab_X^{\bullet}|$, and similarly for $Y$, we obtain an induced map $\Spec f^{\sharp}\colon \Spec \Ab_X^{\bullet} \rightarrow \Spec \Ab_Y^{\bullet}$. Taking geometric realisations in the category $\Sch^{\qcqs}$, we obtain a morphism $f\colon X \rightarrow Y$ by virtue of Theorem \ref{thm:schematic_reconstruction}. 
\end{proof}

Since the category of quasi-compact and quasi-separated schemes is classical, that is a $1$-category, the colimit $|\Spec \Ab_X^{\bullet}|$ can actually be identified with a co-equaliser.

\begin{corollary}\label{cor:schematic_reconstruction}
For any Noetherian scheme $X$ we have a canonical equivalence $$\colim [\Spec \Ab_X^1 \rightrightarrows \Spec \Ab_X^0]  \simeq X,$$ where the co-equaliser is taken in the category of quasi-compact and quasi-separated schemes.
\end{corollary}

The recent paper \cite{bhatt2015tannaka} by Bhatt and Halpern-Leistner extends the aforementioned Tannakian reconstruction result to (spectral) stacks, for which the derived $\infty$-category is compactly generated. 

\begin{corollary}\label{perfect_reconstruction}
If $\Y$ is a spectral stack with quasi-affine diagonal, such that the derived category $\QC(\Y)$ is compactly generated, and $X$ is a Noetherian scheme, then we have
$\Y(X)  \simeq |\Y(\Ab_X^{\bullet})|.$
\end{corollary}

In positive characteristic, stacks of interest rarely satisfy the condition of $\QC(\Y)$ being compactly generated. For instance, it may fail for $BG$, where $G$ is a reductive algebraic group. Nonetheless, as we recalled in Theorem \ref{thm:bhattdhl}, the article\cite{bhatt2015tannaka} also treats the more general case of Noetherian spectral stacks, without the restrictive assumption of compact generation.
However it is no longer sufficient to consider only perfect complexes, additionally the notions of connnectivity and pseudo-coherence become relevant. We will investigate the consequences of their result and adelic descent theory in Paragraph \ref{Noetherian_stacks}.

\subsubsection{Pseudo-coherence}

In this paragraph we study the behaviour of \emph{pseudo-coherence}, also known as \emph{almost perfect complexes}, with respect to the equivalence $\QC(X)  \simeq |\QC(\Ab_X^{\bullet})|$. Recall that an element of $\QC(X)$ is called \emph{pseudo-coherent}, if it can locally be represented by a bounded above complex of finitely generated locally projective $\Oo_X$-modules.

\begin{lemma}\label{lemma:A_conservative}
The functor $\A_X^0(-)\colon \QCoh(X) \rightarrow \Mod(\Oo_X)$ is conservative.
\end{lemma}
\begin{proof}
Assume that $M \in \QCoh(X)$ is a quasi-coherent sheaf on $X$, such that $\A_X^0(M) = 0$. We have $\Gamma(U,M) = \lim_{\Delta} \Gamma(U,\A_X^k(M))$, where the limit is taken in the category of abelian groups. This limit, agrees with the equaliser of $[\Gamma(U,\A_X^0(M)) \rightrightarrows \Gamma(U,\A_X^1(M))]$, and therefore we see that $\Gamma(U,M) = 0$ for all open subsets $U$. This implies $M = 0$.
\end{proof}

\begin{corollary}\label{cor:A_conservative}
If $M \in \QCoh(X)$ is a quasi-coherent sheaf, such that $\A_X^0(M)$ is locally finitely generated, then $M$ is locally finitely generated.
\end{corollary}

\begin{proof}
Assume that $X$ is not locally finitely generated. Then, after replacing $X$ by a suitable affine open subset $U$, we may construct a surjection $\Oo_U^{\oplus I} \twoheadrightarrow M$, with $I$ an infinite set, such that for every finite subset $J \subset I$, the restriction $\Oo_U^{\oplus J} \rightarrow M$ is not surjective. We denote the cokernel of this map by $C_J$.

However, we know that $\A_U^0(M)$ is finitely generated (by Lemma \ref{lemma:loc_fg} a locally finitely generated $\A_U$-module is globally finitely generated). Hence, there must exist a finite subset $J \subset I$, such that $\A_U^0(\Oo_U)^{\oplus J} \twoheadrightarrow \A_U^0(M)$ is a surjection. Therefore, by exactness of $\A_U^0$, we obtain $\A_U^0(C_J) = 0$. However, by Lemma \ref{lemma:A_conservative}, we see that $C_J = 0$, and therefore the map $\Oo_X^{\oplus J} \twoheadrightarrow M$ is a surjection as well. This contradicts our assumption.
\end{proof}

\begin{corollary}\label{cor:pseudo}
Let $M \in \QC(X)$ be such that $M \otimes_{\Oo_X} \A_X^0$ is pseudo-coherent, then $M$ is pseudo-coherent.
\end{corollary}

\begin{proof}
Since $X$ is Noetherian, we only have to show that $\Hc^i(M) = 0$ for $i >> 0$, and that every $\Hc^i(M)$ is locally finitely generated (see \cite[Tag 066E]{stacks-project}). The first assertion follows directly from the Lemmas \ref{lemma:cohomology_sheaves} and \ref{lemma:A_conservative}. Indeed, we know that $\Hc^i(M \otimes_{\Oo_X}\A_X^0)  \simeq \A_X^0(\Hc^i(M))$, and that the functor $\A^0_X(-)$ is conservative. Therefore we see that vanishing of $\Hc^i(M \otimes_{\Oo_X}\A_X^0)$ in high degrees, must imply the same statement for $M$.

To show that all cohomology sheaves $\Hc^i(M)$ are locally finite generated, we may restrict $X$ to an affine open subscheme $U \subset X$. We may assume that $M|_U \otimes_{\Oo_U} \A_U^0$ is a bounded above complex of finitely generated free $\A_U^0$-modules. Therefore there exists a degree $i$, such that $\Hc^j(M \otimes_{\Oo_X} \A_X^0) = 0$ for all $j \geq i + 1$. In particular, we see that $\Hc^i(M \otimes_{\Oo_X} \A_U^i)$ is finitely generated. Since we have that $\Hc^i(M \otimes_{\Oo_X}\A_U^0)  \simeq \A_U^0(\Hc^i(M))$, we see from Corollary \ref{cor:A_conservative} that $\Hc^i(M)$ is finitely generated.

Consider the distinguished triangle $\tau_{\leq i-1}M|_U \rightarrow M|_U \rightarrow \Hc^i(M)[-i]|_U$. Since $U$ is Noetherian, the finitely generated module $\Hc^i$ is pseudo-coherent, and by applying the exact functor $- \otimes_{\Oo_X} \A_X^0$ we obtain that $\tau_{\leq i-1}M \otimes_{\Oo_X} \A_X^0$ is also pseudo-coherent. Hence, we conclude that also $\Hc^{i-1}(\tau_{\leq i-1}M)|_U = \Hc^{i-1}(M)|_U$ is finitely generated. Iterating this argument, we see that all cohomology sheaves are locally finitely generated.
\end{proof}

\begin{corollary}\label{cor:APerf}
We have an equivalence $\APerf(X)  \simeq |\APerf(\Ab_X^{\bullet})|$ for pseudo-coherent complexes.
Similarly, we have an equivalence of symmetric monoidal $\infty$-categories $\APerf_{\cn}(X)  \simeq |\APerf_{\cn}(\Ab_X^{\bullet})_{\otimes}|$, of almost perfect (that is, pseudo-coherent) and connective complexes.
\end{corollary}

\begin{proof}
We have seen that $|\Spec \Ab_X^{\bullet}| \rightarrow X$ is an equivalence in the category $\Sch^{\qcqs}$ of quasi-compact and quasi-separated schemes. In particular we have morphisms of schemes $\Spec \Ab_X^{\bullet} \rightarrow X$, which induce a functor $\APerf(X) \rightarrow \APerf(\Ab_X^{\bullet})$, preserving connective objects, and compatible with the equivalence $\Perf(X)  \simeq |\Perf(\Ab_X^{\bullet})|$.

The adelic realisation functor $-\otimes_{\Oo_X} \A_X^{\bullet}\colon \QC(X) \rightarrow |\QC(\A_X^{\bullet})|$ preserves almost perfect objects; and so does the localisation functor $\Loc\colon \APerf(\Ab_X^{\bullet}) \rightarrow \APerf(\A_X^{\bullet})$, since it is defined as the tensor product $- \otimes_{\Ab_X^{\bullet}} \A_X^{\bullet}$.

We have seen in Lemma \ref{lemma:cohomology_sheaves} that for every complex $M \in \QC(X)$ we have $\A_X^{\bullet}(\Hc^i(M)) = \Hc^i(\A^{\bullet}_X\otimes_{\Oo_X} M)$. Together with Theorem \ref{thm:beilinson} this implies $\Hc^i(M) = |\Hc^i(M \otimes_{\Oo_X} \A_X^{\bullet})|$. And we see that $M$ is connective if and only if $\A_X^{\bullet} \otimes_{\Oo_X} M$ is. 
We have seen in Corollary \ref{cor:pseudo} that $M \in \QC(X)$ is pseudo-coherent, if and only if $\A_X^{\bullet} \otimes_{\Oo_X} M$ is pseudo-coherent. This shows that the image of a pseudo-coherent, and connective $M \in \APerf_{\cn}(\Ab_X^{\bullet})$ lies inside $\APerf_{cn}(X)$. 

We conclude that both directions of the equivalence $\Ind \Perf(X)  \simeq \Ind|\Perf(\Ab_X^{\bullet})|$ respect almost perfect and connective objects. Hence, we obtain the equivalences $\APerf(X)  \simeq |\APerf(\Ab_X^{\bullet})|$ and $\APerf_{\cn}(X)  \simeq |\APerf_{\cn}(\Ab_X^{\bullet})_{\otimes}|$.
\end{proof}

\subsubsection{Ad\`eles and maps to Noetherian stacks}\label{Noetherian_stacks}

We will see that $X$ remains the colimit of the diagram $\Spec \Ab_X^{\bullet}$ in the $\infty$-category of Noetherian stacks.

\begin{theorem}\label{thm:stacky_reconstruction}
Let $X$ be a Noetherian scheme, and $\Y$ a spectral Noetherian stack with quasi-affine diagonal. We then have a canonical equivalence $\Y(X)  \simeq \Y(\Spec \Ab_X^{\bullet})$, that is, $X$ the equivalent to the colimit $|\Spec \Ab_X^{\bullet}|$ in the $\infty$-category of spectral Noetherian stacks with quasi-affine diagonal.
\end{theorem}

\begin{proof}
We have shown in Corollary \ref{cor:APerf} that we have an equivalence $\APerf(X)_{\otimes}  \simeq |\APerf(\Ab_X^{\bullet})_{\otimes}|$. By means of \cite[Lemma 3.12 \& Theorem 1.4]{bhatt2015tannaka} we see that this equivalence realises $X$ as a colimit of $\Spec \Ab_X^{\bullet}$ in the category of spectral Noetherian stacks.
\end{proof}

If $\Y$ is a classical Noetherian stack, that is a groupoid-valued functor $\Aff^{\op} \rightarrow \Gpd$, we may take advantage of the fact that $\Gpd$ form a $2$-category to simplify the limit $|\Y(\Spec \Ab_X^{\bullet})|$.

\begin{corollary}\label{cor:Y}
Let $\Y$ be a Noetherian stack with quasi-affine diagonal, taking values in the $2$-category of groupoids. Then, we have an equivalence of groupoids 
$$\Y(X)  \simeq \lim [\Y(\Ab_X^{0}) \rightrightarrows \Y(\Ab_X^1) \triplerightarrow \Y(\Ab_X^2)],$$
where the limit is taken in the $2$-category of groupoids.
\end{corollary}

\subsubsection{The special case of $G$-bundles}\label{specialcase}

A special case of Corollary \ref{cor:Y} of particular interest to us is the following generalisation of Weil's theorem to arbitrary Noetherian schemes.

\begin{corollary}\label{cor:BG}
Let $G$ be an affine Noetherian group scheme, and $X$ a Noetherian scheme. Then we have a canonical equivalence
$$BG(X)  \simeq \lim [BG(\Ab_X^{0}) \rightrightarrows BG(\Ab_X^1) \triplerightarrow BG(\Ab_X^2)].$$
\end{corollary}

For $G = \GL_n$, and a scheme $U$, $B\mathrm{GL}_n(U)$ is equivalent to the groupoid $\Vect_n(U)$ of rank $n$ vector bundles on $U$. We write $\Vect_n^f(U)$ for full subgroupoid, consisting only of trivial rank $n$ bundles. 

Let $E\in \Vect_n(X)$, since finitely generated projective modules over local rings are free, we may choose a trivialisation $E_x  \simeq \Oo_x^{\oplus n}$, for every $x \in |X|$. This implies that $\Ab_X^0(E)$ is a trivial rank $n$ module over $\Ab_X^0$. Hence, we obtain an equivalence
$$\Vect_n(X)  \simeq \lim [\Vect_n^f(\Ab_X^0) \rightrightarrows \Vect_n^f(\Ab_X^1) \triplerightarrow \Vect_n^f(\Ab_X^2)].$$
For $E \in \Vect_n(X)$ we denote the corresponding objects in $\Vect_n^f(\Ab_X^i)$ by $E^i$.
The choice of a trivialisation $E^0  \simeq (\Ab_X^0)^{\oplus n}$ induces two trivialisations $\psi_i\colon E^1  \simeq (\Ab_X^1)^{\oplus n}$ for $i=0\text{ ,}1$. The trivialisations are obtained by means of the defining property of a cartesian module. Let $\partial_i\colon \{i\} \hookrightarrow \{0,1\} = [1]$ in $\Delta$. Since $E^{\bullet}$ is a cartesian $\Ab_X^{\bullet}$-module, we have $E^1  \simeq E^0 \otimes_{\Ab_X^0,\partial_i} \Ab_X^1$. We let $\phi\in \GL_n(\Ab_X^1)$ be the linear map $\phi = \psi_0 \circ \psi_1^{-1}\colon (\Ab_X^1)^{\oplus n} \rightarrow (\Ab_X^1)^{\oplus n}$.

For $0 \leq i < j\leq 2$, we denote by $\partial_{ij}\colon \{i,j\} \hookrightarrow \{0, 1, ,2\}= [2]$ in $\Delta$. We denote by $\phi_{ij} = \phi \otimes_{\Ab_X^1,\partial_{ij}} \Ab_X^2 \in \GL_n(\Ab_X^2)$. By construction of $\phi$ the cocycle identity $\phi_{02} = \phi_{01} \circ \phi_{12}$ is satisfied.

Vice versa, every $\phi \in \GL_n(\Ab_X^1)$, satisfying the cocycle identity $\phi_{02} = \phi_{01} \circ \phi_{12}$, corresponds to an object $E^{\bullet} \in \lim [\Vect_n^f(\Ab_X^0) \rightrightarrows \Vect_n^f(\Ab_X^1) \triplerightarrow \Vect_n^f(\Ab_X^2)]$, plus the choice of a trivialisation $E^0  \simeq (\Ab_X^0)^{\oplus n}$. Summarising we obtain interpretation of Corollary \ref{cor:BG}.

\begin{corollary}\label{cor:BG2}
Let $X$ be a Noetherian scheme and $G$ a special group scheme (that is, every $G$-bundle on a Noetherian scheme is Zariski-locally trivial). We denote by $G(\Ab_X^1)^{\text{cocycle}}$ the subset consisting of $\phi \in \G(\Ab_X^1)$ satisfying the cocycle condition $\phi_{02} = \phi_{01}\circ \phi_{12}$ in $G(\Ab_X^2)$. Then, we have an equivalence of groupoids
$BG(X)  \simeq [G(\Ab_X^1)^{\text{cocycle}}/G(\Ab_X^0)].$
\end{corollary}

When $X$ is a curve, Weil's formulation $BG(X)  \simeq [G(F_X)\setminus G(\Ab_X)\;/\;G(\mathbb{O}_X)]$ can be directly deduced from the aforementioned one. Recall that $\Ab_X^1 = F_X \times \mathbb{O}_X \times \mathbb{A}_X$, and that $\Ab_X^2 = \Ab_X^1$ for dimension reasons. Therefore, $G(\Ab_X^1)^{\text{cocycle}} = G(\Ab_X^1)$, and we obtain the quotient
$$[G(F_X) \times G(\mathbb{O}_X) \times G(\mathbb{A}_X) \;/\; G(F_X) \times G(\mathbb{O}_X)]  \simeq [G(F_X)\setminus G(\Ab_X)\;/\;G(\mathbb{O}_X)].$$ 

\begin{rmk}
In \cite[Section 5.2]{MR1138291} Huber defines \emph{rational ad\`eles} $a_X^{\bullet}$. The arguments of this paper also apply to $a_X^{\bullet}$, in fact the proofs can be slightly simplified in this case.
\end{rmk}

\appendix
\section{Results from $\infty$-category theory}

An $\infty$-category $\C$ is called \emph{small}, if it is equivalent to an $\infty$-category modelled by a quasi-category whose underlying simplicial set is small. If $\C$ is a classical category, this is tantamount to $\C$ being equivalent to a category, for which the class of objects, and morphism sets for every pair of objects, form small sets.

We say that $\C$ is \emph{locally small}, if for every pair of objects $X$, $Y$ the mapping space $\Map(X,Y)$ can be modelled by a small Kan complex (see \cite[Definition 5.4.1.7]{Lurie:bh}). For a classical category $\C$ this amounts to the condition that $\Map(X,Y)$ is a small set.

We will only consider locally small $\infty$-categories. Because we are sometimes taking limits of $\infty$-categories (over small diagrams), it is important to observe that this operation will preserve local smallness.

\begin{rmk}\label{rmk:loc_small}
Let $\C_-\colon K \rightarrow \inftyCat$ be a diagram of $\infty$-categories, such that $K$ is a small simplicial set, and for each $k \in K$ we have that $\C_k$ is a locally small $\infty$-category. Then, the limit $\lim_{k \in K} \C_k$ is also a locally small $\infty$-category.
\end{rmk}

The reader may easily verify this assertion with the help of \cite[Corollary 3.3.3.2]{Lurie:bh}, which identifies the limit $\C = \lim_{k \in K} \C_k$ with the $\infty$-category of cartesian sections, of a cartesian fibration $\widetilde{\C} \rightarrow K$, corresponding to the functor $\C_-$. This implies that the mapping spaces in $\C$ are equivalent to the mapping spaces in the $\infty$-category of cartesian sections of $\widetilde{\C} \rightarrow K$. If $X, Y\colon K \rightarrow  \widetilde{\C}$ are two such sections, then we obtain $$\Map(X,Y) \simeq \lim_{k \in K} \Map(X_k,Y_k).$$ Since small limits of small spaces remain small, we conclude that $\C$ is locally small. The description of mapping spaces in limits of $\infty$-categories will be recorded for future reference.

\begin{lemma}\label{lemma:limits}
Let $\C_-\colon K \rightarrow \C$ be a diagram of $\infty$-categories, parametrised by a small simplicial set $K$. We denote by $\C = \lim_{k \in K} \C_k$ the limit of this diagram. Given objects $X\text{, }Y \in \C$, we have a canonical equivalence of spaces $\Map(X,Y)  \simeq \lim_{k \in K}(X_k,Y_k).$
\end{lemma}

The following result is \cite[Corollary 5.5.2.9 \& Remark 5.5.2.10]{Lurie:bh}.

\begin{theorem}[Lurie's Adjoint Functor Theorem]\label{thm:adjointfunctor}
Let $F\colon \C \rightarrow \D$ be a functor between stable $\infty$-categories. If $\C$ is presentable, and $\D$ locally small, then $F$ admits a right adjoint $G\colon \D \rightarrow \C$, if and only if $F$ preserves small colimits.
\end{theorem}

We refer the reader to \cite[Proposition 1.4.4.1.]{Lurie:ha} for a proof of the following result.

\begin{proposition}\label{prop:colimits}
If $\C$ is a stable $\infty$-category, admitting arbitrary coproducts, then $\C$ possesses all small colimits. Moreover, let $F\colon \C \rightarrow \D$ be an exact functor between cocomplete stable $\infty$-categories, commuting with small coproducts. Then, $F$ commutes with small colimits.
\end{proposition}

The notion of compact objects in stable $\infty$-categories will be essential to us. There are two equivalent characterisations of compactness. We refer the reader to \cite[Proposition 1.4.4.1.]{Lurie:ha} for a proof of equivalence.

\begin{definition}
Let $\C$ be a stable $\infty$-category. An object $X \in \C$ is called \emph{compact}, if $\Hom(X,-)$ commutes with small coproducts. Equivalently, for $(Y_i)_{i \in I} \in \C$, a small family of objects in $\C$, for every morphism $X \rightarrow \bigoplus_{i \in I} Y_i$, there exists a finite subset $J \subset I$, such that the map factors through $\bigoplus_{i \in J} Y_i \rightarrow \bigoplus_{i \in I} Y_i$.
\end{definition}

The full subcategory of compact objects of $\C$ will be denoted by $\C^c$. We say that $\C$ is \emph{compactly generated}, if $\C  \simeq \Ind \C^c$.

\begin{lemma}\label{lemma:generators_coproducts}
Let $F\colon \C \rightarrow \D$ be an exact functor between cocomplete stable $\infty$-categories, with a right adjoint $G$. If $\C$ is compactly generated, and $F$ preserves compact objects, then $G$ commutes with small colimits.
\end{lemma}

\begin{proof}
By Proposition \ref{prop:colimits} it suffices to show that $G$ commutes with small coproducts. That is, for $\{Y_i\}_{i \in I}$, where $I$ is a small set indexing objects in $\D$, we have to show that the natural map $\bigoplus_{i \in I} G(Y_i) \rightarrow G(\bigoplus_{i \in I} Y_i)$ is an equivalence. Since $\C$ is compactly generated, this is equivalent to
$$\Hom(X,\bigoplus_{i \in I} G(Y_i)) \rightarrow \Hom(X,G(\bigoplus_{i \in I} Y_i))$$
being an equivalence for every compact object $X \in \C$. Compactness of $X$, and the adjunction between $F$ and $G$ allows us to show instead that the map
$$\bigoplus_{i \in I}\Hom(F(X),Y_i) \rightarrow \bigoplus_{i \in I} \Hom(F(X),Y_i)$$
is an equivalence. But this is a morphism equivalent to the identity map.
\end{proof}

The following is a well-known criterion for an adjunction to be an equivalence.

\begin{lemma}\label{lemma:weakBB}
Let $F\colon \C \rightarrow \D$ be a functor between $\infty$-categories with right adjoint $G$. If $F$ is fully faithful and $G$ is conservative, then $F$ and $G$ are mutually inverse equivalences.
\end{lemma}

\begin{proof}
Since $F$ is fully faithful, the unit $\id_{\C} \rightarrow GF$ is an equivalence of functors. It remains to check that the co-unit $F G \rightarrow \id_{\D}$ is an equivalence. By assumption, $G$ is conservative, therefore it suffices to prove that $G F G \rightarrow G$ is an equivalence of functors. By the standard properties of adjunctions, the composition
$$G \simeq \id_{\C} G  \simeq G  F  G \rightarrow G  \id_{\D} \simeq G$$
is equivalent to the identity map. Thus the co-unit $FG \rightarrow \id_{\D}$ is an equivalence, and we conclude the proof.
\end{proof}

Let $\D$ be a small stable $\infty$-category, admitting finite limits. We will define define a functor 
$$\Map(-,-)\colon \Delta^{\op} \times \D^{\Delta} \rightarrow \D^{\Delta},$$
sending a co-simplicial object $U^{\bullet}$ in $\D$, and $[n] \in \Delta$, to a co-simplicial object $\Map([n],U^{\bullet})$. Recall that we have a category $\Delta/[n]$, whose objects are morphisms $[k] \rightarrow [n]$. It is endowed with a forgetful functor to $\Delta$, and this construction is contravariantly functorial in $[n]$. The co-simplicial object $\Map([n],U^{\bullet})$ is given by considering the functor $\Delta/[n] \rightarrow \Delta \rightarrow \D$, and taking the fibrewise limit along $\Delta/[n] \rightarrow \Delta$ to obtain a functor $\Map([n],U^{\bullet})\colon\Delta \rightarrow \D$.

\begin{definition}\label{defi:infty_homotopy}
Let $U^{\bullet}\text{, }V^{\bullet}\colon \Delta \rightarrow \D$ be two co-simplicial objects in a small $\infty$-category $\D$, admitting finite limits. A homotopy between $f\text{, g}\colon U^{\bullet} \rightarrow V^{\bullet}$ is a morphism $h\colon U^{\bullet} \rightarrow \Map([1],V^{\bullet})$, such that $\mathrm{ev}_0 \circ h \simeq f$, and $\mathrm{ev}_1 \circ h \simeq g$. Here $\mathrm{ev}_i\colon \Map([1],V^{\bullet}) \rightarrow V^{\bullet}$ denotes the evaluation map, induced by the inclusion $[0] \simeq \{i\} \hookrightarrow [1]$.
\end{definition}

We refer the reader to \cite[Tag 019U]{stacks-project} for a detailed explanation for how this recovers the explicit definition given earlier, in the case that $\D$ is a classical category (with small colimits). The reason for considering the notion of homotopy between morphisms of simplicial objects, is the following lemma.

\begin{lemma}\label{lemma:cat_homotopy}
Let $\C^{\bullet}$ be a co-simplicial small $\infty$-category. Then, the map $c\colon \C^{\bullet} \rightarrow \Map([1],\C^{\bullet})$, corresponding to the map $s\colon [1] \rightarrow [0]$ of simplicial sets, induces an equivalence of the $\infty$-categories obtained by totalisation $|\C^{\bullet}| \rightarrow |\Map([1],\C^{\bullet})|$. In particular we see that two homotopy equivalent co-simplicial small $\infty$-categories have equivalent totalisations.
\end{lemma}

\begin{proof}
For co-simplicial objects in the full subcategory of small $\infty$-groupoids $\Grpd$ this assertion is well-known, and follows for instance from Meyer's \cite[Theorem 2.4]{meyer1990cosimplicial}, or Bousfield's \cite[Proposition 2.13]{bousfield2003cosimplicial}. We will apply their result in our proof of the analogous assumption for co-simplicial objects in small $\infty$-categories.

We will show first that the functor $|c|\colon |\C^{\bullet}| \hookrightarrow |\Map([1],\C^{\bullet})|$ is fully faithful. Given $X,\text{, Y} \in |\C^{\bullet}|$, we recall Lemma \ref{lemma:limits}, which asserts that $\Map(X,Y)$ is itself the totalisation of a co-simplicial object in $\Grpd$, $\Map(X^{\bullet},Y^{\bullet})$. It follows directly from the definitions that $\Map(c(X)^{\bullet},c(Y)^{\bullet}) \simeq |\Map([1],\Map(X^{\bullet},Y^{\bullet}))|$. We may therefore conclude that $\Map(X,Y) \simeq \Map(c(X),c(Y))$, and thus that $c$ is fully faithful.

It remains to show that $c$ is essentially surjective. For a small $\infty$-category $\D$, we denote by $\D^{\times}$ the maximal $\infty$-groupoid in $\D$, obtained by discarding all non-equivalences in $\D$. In fact we have a functor $(-)^{\times} \colon \inftyCat \rightarrow \Grpd$. We therefore see, that the morphism of co-simplicial objects in $\Grpd$, 
$$c^{\times}\colon (\C^{\bullet})^{\times} \rightarrow \Map([1],\C^{\bullet})^{\times} \simeq \Map([1],(\C^{\bullet})^{\times})$$ 
induces an equivalence after totalisation. This implies that $c$ is essentially surjective.
\end{proof}

\section{Notation}

\begin{tabular}{p{1.75cm}p{10cm}}
$|X|$ & underlying topological space of a scheme $X$ \\
$|X|_{\bullet}$ & simplicial set of chains of points in $|X|$, ordered by specialisation \\
$\Ab_{X}^{\bullet}$ & co-simplicial ring of ad\`eles \\
$\A_X^{\bullet}$ & co-simplicial sheaf of algebras of ad\`eles \\
$\Sch^{\N}$ & category of Noetherian schemes \\
$\Sch^{\qcqs}$ & category of quasi-compact and quasi-separated schemes \\
$BG$ & stack of $G$-bundles \\
$\Mod(R)$ & category of $R$-modules, where $R$ is a ring, or a sheaf of rings \\
$\P(R)$ & exact category of projective $R$-modules ($R$ is a ring) \\
$\P(\A)$ & exact category of direct summands of free $\A$-modules ($\A$ is a sheaf of algebras) \\
$\D^-(\A)$ & see Definition \ref{defi:Dminus}\\
$\D(\C)$ & derived $\infty$-category of an exact category $\C$ \\
$\DMod(\A)$ & derived $\infty$-category of $\Mod(\A)$, where $\A$ is a ring or a sheaf of algebras \\
$\Perf(\A)$ & full subcategory of $\D(\Mod(\A))$ consisting of perfect complexes \\
$\QC(\A)$ & $\Ind \Perf(\A)$. If $\A$ is a ring (on a point), it is equivalent to $\D(\Mod(\A))$ \\
$\C^c$ & full subcategory of compact objects of a stable $\infty$-category $\C$ \\
$\APerf(\A)$ & full subcategory of $\QC(\A)$ corresponding to pseudo-coherent complexes \\
$\inftyCat$ & $\infty$-category of small $\infty$-categories \\
$\Cat$ &  $2$-category of small categories \\
$\Grpd$ & the $\infty$-category of small $\infty$-groupoids, or Kan complexes, or spaces\\
$\Gpd$ & the $2$-category of groupoids \\
$\St$ & the $\infty$-category of small stable $\infty$-categories \\
$\Map_{\C}(X,Y)$ & mapping space between two objects $X$ and $Y$ in an $\infty$-category $\C$\\
$\Hom_{\C}(X,Y)$ & mapping spectrum between two objects $X$ and $Y$ in a stable $\infty$-category $\C$, or chain complex, if $\C$ is $R$-linear\\
$(-)_{\otimes}$ & subscript used to refer to symmetric monoidal $\infty$-categories, or also $\infty$-categories of symmetric monoidal $\infty$-categories\\
$|X^{\bullet}|$ & limit (that is, geometric realisation, or totalisation) of the co-simplicial diagram $X^{\bullet}$ in an $\infty$-category $\C$ \\
$|X_{\bullet}|$ & colimit (that is, geometric realisation, or totalisation) of the simplicial diagram $X_{\bullet}$ in an $\infty$-category $\C$ \\
\end{tabular}

\bibliographystyle{amsalpha}
\bibliography{adeles.bib}

\providecommand{\bysame}{\leavevmode\hbox to3em{\hrulefill}\thinspace}
\providecommand{\MR}{\relax\ifhmode\unskip\space\fi MR }
\providecommand{\MRhref}[2]{%
  \href{http://www.ams.org/mathscinet-getitem?mr=#1}{#2}
}
\providecommand{\href}[2]{#2}
\begin{thebibliography}{{Wei}38a}

\bibitem[BBT13]{MR3003930}
Oren Ben-Bassat and Michael Temkin, \emph{Berkovich spaces and tubular
  descent}, Adv. Math. \textbf{234} (2013), 217--238. \MR{3003930}

\bibitem[Bei80]{MR565095}
A.~Beilinson, \emph{Residues and ad{\`e}les}, Funktsional. Anal. i Prilozhen.
  \textbf{14} (1980), no.~1, 44--45. \MR{565095 (81f:14010)}

\bibitem[BGW]{Braunling:2014aa}
O~Braunling, M~Groechenig, and J~Wolfson, \emph{A {G}eneralized
  {C}ontou-{C}arr{\`e}re {S}ymbol and its {R}eciprocity {L}aws in {H}igher
  {D}imensions, eprint}, arXiv preprint arXiv:1410.3451.

\bibitem[Bha14]{Bhatt:2014aa}
Bhargav Bhatt, \emph{Algebraization and {T}annaka duality}, arXiv preprint
  arXiv:1404.7483 (2014).

\bibitem[BHL15]{bhatt2015tannaka}
Bhargav Bhatt and Daniel Halpern-Leistner, \emph{Tannaka duality revisited},
  arXiv preprint arXiv:1507.01925 (2015).

\bibitem[BL94]{MR1289330}
Arnaud Beauville and Yves Laszlo, \emph{Conformal blocks and generalized theta
  functions}, Comm. Math. Phys. \textbf{164} (1994), no.~2, 385--419.
  \MR{1289330 (95k:14011)}

\bibitem[BL95]{MR1320381}
\bysame, \emph{Un lemme de descente}, C. R. Acad. Sci. Paris S{\'e}r. I Math.
  \textbf{320} (1995), no.~3, 335--340. \MR{1320381 (96a:14049)}

\bibitem[Bou03]{bousfield2003cosimplicial}
AK~Bousfield, \emph{Cosimplicial resolutions and homotopy spectral sequences in
  model categories}, Geometry \& Topology \textbf{7} (2003), no.~2, 1001--1053.

\bibitem[BS13]{bhatt2013pro}
Bhargav Bhatt and Peter Scholze, \emph{The pro-\'etale topology for schemes},
  arXiv preprint arXiv:1309.1198 (2013).

\bibitem[B{\"u}h10]{MR2606234}
Theo B{\"u}hler, \emph{Exact categories}, Expo. Math. \textbf{28} (2010),
  no.~1, 1--69. \MR{2606234 (2011e:18020)}

\bibitem[Efi10]{Efimov:2010fk}
A.I. Efimov, \emph{Formal completion of a category along a subcategory},
  arXiv:1006.4721, 06 2010.

\bibitem[Fre07]{frenkel2007lectures}
Edward Frenkel, \emph{Lectures on the {L}anglands program and conformal field
  theory}, Frontiers in number theory, physics, and geometry II, Springer,
  2007, pp.~387--533.

\bibitem[GN43]{gelfand}
I.M. Gelfand and M.A. Naimark, \emph{On the embedding of normed linear rings
  into the ring of operators in {H}ilbert space}, Mat. Sbornik \textbf{12}
  (1943), no.~197-213.

\bibitem[GP]{garlandgeometry}
H~Garland and M~Patnaik, \emph{Geometry of loop {E}isenstein series}.

\bibitem[Hub91]{MR1138291}
A.~Huber, \emph{On the {P}arshin-{B}e\u\i linson ad{\`e}les for schemes}, Abh.
  Math. Sem. Univ. Hamburg \textbf{61} (1991), 249--273. \MR{1138291
  (92k:14024)}

\bibitem[Kel96]{keller1996derived}
Bernhard Keller, \emph{Derived categories and their uses}, Handbook of algebra
  \textbf{1} (1996), 671--701.

\bibitem[Lie05]{lieblich2005moduli}
Max Lieblich, \emph{Moduli of complexes on a proper morphism}, arXiv preprint
  math/0502198 (2005).

\bibitem[Lur]{Lurie:ha}
Jacob Lurie, \emph{Higher algebra.}, Preprint, available at http://www. math.
  harvard. edu/\~{} lurie.

\bibitem[Lur07]{Lurie:bh}
\bysame, \emph{Higher topos theory}, available online from the author's
  homepage: http://www-math. mit. edu/\~{} lurie/papers/highertopoi. pdf
  (2007).

\bibitem[Lur11]{lurie2011quasi}
\bysame, \emph{Quasi-coherent sheaves and {T}annaka duality theorems}, preprint
  available at http://www. math. harvard. edu/lurie (2011).

\bibitem[Mey90]{meyer1990cosimplicial}
Jean-Pierre Meyer, \emph{Cosimplicial homotopies}, Proceedings of the American
  Mathematical Society \textbf{108} (1990), no.~1, 9--17.

\bibitem[Mor]{morrow}
Matthew Morrow, \emph{An introduction to higher dimensional local fields and
  ad\`eles}, arXiv:1204.0586.

\bibitem[{Sta}]{stacks-project}
{Stacks Project Authors}, \emph{Stacks project},
  http://math.columbia.edu/algebraic\_geometry/stacks-git.

\bibitem[TT90]{MR1106918}
R.~W. Thomason and T.~Trobaugh, \emph{Higher algebraic {$K$}-theory of schemes
  and of derived categories}, The {G}rothendieck {F}estschrift, {V}ol.\ {III},
  Progr. Math., vol.~88, Birkh{\"a}user Boston, Boston, MA, 1990, pp.~247--435.
  \MR{1106918 (92f:19001)}

\bibitem[TV08]{MR2394633}
Bertrand To{{\"e}}n and Gabriele Vezzosi, \emph{Homotopical algebraic geometry.
  {II}. {G}eometric stacks and applications}, Mem. Amer. Math. Soc.
  \textbf{193} (2008), no.~902, x+224. \MR{2394633 (2009h:14004)}

\bibitem[{Wei}38a]{zbMATH03028558}
Andr\'e {Weil}, \emph{{G\'en\'eralisation des fonctions ab\'eliennes.}}, {J.
  Math. Pures Appl. (9)} \textbf{17} (1938), 47--87 (French).

\bibitem[Wei38b]{weil1938algebraischen}
Andr{\'e} Weil, \emph{Zur algebraischen theorie der algebraischen
  {F}unktionen.}, Journal f{\"u}r die reine und angewandte Mathematik
  \textbf{179} (1938), 129--133.

\bibitem[Wit88]{witten1988quantum}
Edward Witten, \emph{Quantum field theory, {G}rassmannians, and algebraic
  curves}, Communications in Mathematical Physics \textbf{113} (1988), no.~4,
  529--600.

\end{thebibliography}
\end{document}